\documentclass[twoside, final]{amsart} 

\usepackage{amsmath, amsfonts, amsthm, amssymb}
\usepackage{exscale}
\usepackage{cite}
\usepackage{epsfig}
\usepackage{amscd}
\usepackage{graphics}
\usepackage{color}
\usepackage{dsfont}

\usepackage[notref, notcite]{showkeys}

\renewcommand{\geq}{\geqslant}
\renewcommand{\leq}{\leqslant}

\newtheorem{theorem}{Theorem}

\newtheorem{proposition}{Proposition}[section]

\newtheorem{corollary}[proposition]{Corollary}
\newtheorem{lemma}[proposition]{Lemma}
\newtheorem*{main-theorem}{Main Theorem}
\newtheorem*{theorem*}{Theorem}

\theoremstyle{definition}

\newtheorem{remark}[proposition]{Remark}

\newtheorem*{remark*}{Remark}

\numberwithin{equation}{section}

\def\phi{\varphi}

\def\ZZ{{\mathbb Z}}
\def\NN{{\mathbb N}}
\def\reals{{\mathbb R}}
\def\cx{{\mathbb C}}

\def\Ci{{\mathcal C}^\infty}
\def\Re{\,\mathrm{Re}\,}
\def\Im{\,\mathrm{Im}\,}

\def\sgn{\mathrm{sgn}\,}

\def\WF{\mathrm{WF}}

\def\supp{\mathrm{supp}\,}

\def\O{{\mathcal O}}
\def\SS{{\mathbb S}}
\def\s{{\mathcal S}}

\def\Op{\mathrm{Op}\,}

\def\csh{{\left( h/\tilde{h} \right)}}
\def\phi{\varphi}

\def\dist{\text{dist}\,}

\def\be{\begin{eqnarray*}}
\def\ee{\end{eqnarray*}}
\def\ben{\begin{eqnarray}}
\def\een{\end{eqnarray}}

\def\lll{\left\langle}
\def\rrr{\right\rangle}

\def\L2R{L_{\text{Rest}}^2}

\def\11{\mathds{1}}

\def\HH{\mathcal{H}}

\def\RR{\mathbb{R}}
\def\tpsi{\tilde{\psi}}
\def\tchi{\tilde{\chi}}
\def\L2c{L^2_{\text{comp}}}

\def\th{\tilde{h}}

\def\tDelta{\widetilde{\Delta}}
\def\tV{\widetilde{V}}

\def\tP{\widetilde{P}}
\def\tp{\tilde{p}}

\def\tB{\tilde{B}}

\def\tGamma{\widetilde{\Gamma}}

\def\tR{\tilde{R}}

\def\Vol{\text{Vol}}

\def\tQ{\widetilde{Q}}

\def\tE{\widetilde{E}}
\def\p{\partial}

\def\Vmin{V_{\text{min}}}
\def\tLambda{\widetilde{\Lambda}}
\def\GG{\mathcal{G}}
\def\tq{\tilde{q}}

\newcommand{\abs}[1]{{\left\lvert{#1}\right\rvert}}
\newcommand{\norm}[1]{{\left\lVert{#1}\right\rVert}}
\newcommand{\ang}[1]{{\left\langle{#1}\right\rangle}}
\newcommand{\pa}{{\partial}}
\newcommand{\ep}{{\epsilon}}
\newcommand{\hamvf}{{\textsf{H}}}

\newcommand{\B}{\mathcal{B}}

\begin{document}

\title[Resolvent estimates]{High-frequency resolvent estimates on
  asymptotically Euclidean warped products}

\author[H. Christianson]{Hans Christianson}

\email{hans@math.unc.edu}
\address{Department of Mathematics, UNC-Chapel Hill \\ CB\#3250
  Phillips Hall \\ Chapel Hill, NC 27599}

\subjclass[2010]{35B34,35S05,58J50,47A10}
\keywords{}

\begin{abstract}
We consider the resolvent on asymptotically Euclidean warped product manifolds
in an appropriate 0-Gevrey class, with trapped sets consisting of only
finitely many components.  We prove that the high-frequency resolvent
is either bounded by $C_\epsilon |\lambda|^\epsilon$ for any $\epsilon>0$, or blows
up faster than any polynomial (at least along a subsequence).  A
stronger result holds if the manifold is analytic.  The
method of proof is to exploit the warped product structure to separate
variables, obtaining a one-dimensional semiclassical Schr\"odinger operator.  We
then classify the microlocal resolvent behaviour associated to every possible type of critical
value of the potential, and translate this into the associated
resolvent estimates.  Weakly stable trapping admits highly
concentrated quasimodes and fast growth of the resolvent.  Conversely,
using a
delicate inhomogeneous blowup procedure loosely based on the classical
positive commutator argument, we show that any weakly
unstable trapping forces at least some spreading of quasimodes.

As a first application, we conclude that either there is a resonance free region of size $|
\Im \lambda | \leq C_\epsilon | \Re \lambda |^{-1-\epsilon}$ for any
$\epsilon>0$, or there is a sequence of resonances converging to the
real axis faster than any polynomial.  Again, a stronger result holds
if the manifold is analytic.  As a second application, we
prove a spreading result for weak quasimodes in partially rectangular billiards.

\end{abstract}

\maketitle

\section{Introduction}
\label{S:intro}

In this paper, we consider manifolds which have a warped product
structure and are 
asymptotically Euclidean with a certain 0-Gevrey regularity.  Our
main result is that the cutoff resolvent is either (almost) bounded
or blows up faster than any polynomial.  Of course, the proof gives
much more information than this simple statement, but for aesthetic
reasons we prefer to phrase it in this fashion.  Let us state the main
result.

\begin{theorem}
\label{T:dichotomy}
Let $X$ be a 0-Gevrey smooth $\GG^0_\tau$, $\tau < \infty$, warped product manifold without boundary
which is a short range perturbation of Euclidean space (with one or
two infinite ends).  Assume also that the trapped set on $X$ has
finitely many connected components.  Let $-\Delta$ be the
Laplace-Beltrami operator on $X$.  

Then either

{\bf 1:}  For every $\phi \in \Ci_c(X)$ and every $\epsilon>0$, there exists $C_\epsilon
>0$ such that
\begin{equation}
\label{E:lambda-eps}
\| \phi (-\Delta - (\lambda - i0)^2)^{-1} \phi \| \leq C_\epsilon | \lambda |^{\epsilon},
\,\,\, | \lambda | \gg 1.
\end{equation}
or

{\bf 2:}  For every $N>0$, there exists $\phi \in \Ci_c(X)$, $C_N >0$, and a sequence $\lambda_j \in
\reals$, $\lambda_j \to \infty$, such that 
\begin{equation}
\label{E:lambda-bup-1}
\| \phi (-\Delta - (\lambda_j - i0)^2)^{-1} \phi \| \geq C_N |\lambda_j|^N.
\end{equation}

\end{theorem}

\begin{remark}
The warped product structure at infinity can be replaced by a number
of different non-trapping infinite ``ends'', using the recent gluing
theorem of Datchev-Vasy \cite{DaVa-gluing} (see also \cite{Chr-disp-1}
and Appendix \ref{A:gluing}).  

The dichotomy in Theorem
\ref{T:dichotomy} is from the following idea: if there is any weakly
stable trapping on $X$, then there are well-localized quasimodes, and
we are in Case 2 of the Theorem.  If all the trapping is at least
weakly unstable, we need to prove there is weak microlocal
non-concentration near each connected component of the trapped set, as
well as prove that there is no strong tunneling between different
connected components (i.e. that the different connected components at
the same energy don't ``talk'' to each other too much).

As we shall see, the worst behaviour in \eqref{E:lambda-eps} in Theorem \ref{T:dichotomy}
comes from weakly unstable trapping which is infinitely
degenerate.  Since such trapping cannot occur on an analytic manifold,
there is a nice improvement in this case, given in the next Corollary.  

\end{remark}

\begin{corollary}
\label{C:little-o}
In addition to the assumptions of Theorem \ref{T:dichotomy}, assume
that the manifold $X$ is analytic.  

Then either

{\bf 1:}  There exists $\delta>0$ such that for every $\phi \in
\Ci_c(X)$, there is a constant $C>0$ for which  we have
the estimate 
\begin{equation}
\label{E:lambda-delta}
\| \phi (-\Delta - (\lambda - i0)^2)^{-1} \phi \| \leq C | \lambda |^{-\delta},
\,\,\, | \lambda | \to \infty.
\end{equation}
or

{\bf 2:}  For every $N>0$, there exists $\phi \in \Ci_c(X)$, $C_N>0$, and a sequence $\lambda_j \in
\reals$, $\lambda_j \to \infty$, such that 
\begin{equation}
\label{E:lambda-bup-2}
\| \phi (-\Delta - (\lambda_j - i0)^2)^{-1} \phi \| \geq C_N |\lambda_j|^N.
\end{equation}

\end{corollary}

\begin{remark} 

Upon rescaling to a semiclassical problem, Theorem \ref{T:dichotomy}
states that the semiclassical cutoff resolvent is either controlled by
$h^{-2-\epsilon}$ for any $\epsilon>0$, or blows up faster than
$h^{-N}$ for any $N$.  The 
corollary states that if the manifold is analytic, the first possibility can be replaced with
$h^{-2 + \delta}$ for some $\delta>0$ fixed, depending on the trapping.

\end{remark}

As usual, high energy resolvent estimates imply there are regions free
of resonances by simple perturbation of the spectral parameter.  On
the other hand, the proof of the alternative large growth of the
resolvent along a subsequence proceeds by quasimode construction.
Then 
if our metric has a complex analytic extension outside
of a compact set, we can apply the results of Tang-Zworski \cite{TaZw}
to conclude existence of resonances.  This results in the following
Corollary.

\begin{corollary}
In addition to the assumptions of Theorem 
\ref{T:dichotomy}, assume $X$ admits a complex analytic extension
outside of a compact set so that resonances may be defined by complex
scaling.  Then either

{\bf 1:} For every $\epsilon>0$ there is a constant $C_\epsilon>0$ such that
the region
\[
\{ \lambda \in \cx : | \Im \lambda | \leq C_\epsilon | \Re \lambda
|^{-1-\epsilon}, \,\, | \lambda | \gg 1 \}
\]
is free of resonances and the estimate \eqref{E:lambda-eps} holds
there (with a suitably modified constant), or

{\bf 2:} For every $N>0$, there exists a sequence of resonances $\{
\lambda_j \}$ such that
\[
| \Im \lambda_j | \leq | \Re \lambda_j |^{-N}, \,\, | \lambda_j | \to
\infty.
\]

In particular, if $X$ is analytic, then either there exists $\delta>0$ and
$C>0$ such that the region 
\[
\{  \lambda \in \cx : | \Im \lambda | \leq C | \Re \lambda
|^{-1+\delta}, \,\, | \lambda | \gg 1 \}
\]
is free from resonances, or there is a sequence converging to the real
axis at an arbitrarily fast polynomial rate.

\end{corollary}

\subsection{Resolvents and the local smoothing effect}
One of the many motivations for studying resolvents and resolvent
estimates is to understand the local smoothing effect for the
Schr\"odinger equation on manifolds with trapping.  It is well known
(see, for example, \cite{Tao-book,Doi}) that on asymptotically
Euclidean manifolds without trapping, solutions to the Schr\"odinger
equation enjoy a $1/2$ derivative local smoothing effect.  This says
that, locally in space, and on average in time, solutions are $1/2$
derivative smoother than the initial data.  To be precise, let $X$ be
such a manifold, $-\Delta$ the Laplacian on $X$, $u_0$ a Schwartz
function on $X$, and $\chi \in \Ci_c(X)$ a cutoff function.  Then the following estimate holds true for any $T>0$:
\[
\int_0^T \| \chi e^{i t \Delta} u_0 \|_{H^{1/2}(X)}^2 dt \leq C_T \|
u_0 \|_{L^2(X)}^2.
\]
There are several ways to prove such an estimate; one way proceeds
through resolvent estimates (see Section \ref{S:l-sm-pf-sec} below).
A nice benefit of using the resolvent formalism to understand local in
time local smoothing (that is, for finite $T$) is that one really sees
how the spectral estimates are related to the smoothing effect.  Since
one only needs a resolvent estimate in a fixed strip near the real
axis, if one is in a situation where the limiting resolvent blows up,
one simply uses the trivial bound away from the real axis to get a
zero derivative smoothing effect (or just integrates the $L^2(X)$ mass
in time).  However, if the limiting resolvent has some decay, then
there is a non-trivial local smoothing estimate.  This is the case,
for example, if the manifold is analytic and all of the trapping is
at least weakly unstable.  Let us state this as a corollary:

\begin{corollary}
\label{C:l-sm-cor}

Let $X$ be an analytic warped product manifold so that all of the assumptions of Corollary
\ref{C:little-o} hold.  Assume also that every connected component of
the trapped set is at least weakly unstable, so that conclusion $1$ of
Corollary \ref{C:little-o} holds for some $\delta>0$.  Then for all
$\chi \in \Ci_c(X)$, $u_0 \in \s (X)$, and $T>0$, there exists $C_T>0$
such that
\[
\int_0^T \| \chi e^{it \Delta} u_0 \|_{H^{\delta/2}(X)}^2 dt \leq C_T
\| u_0 \|_{L^2(X)}^2.
\]

\end{corollary}

\subsection*{Acknowledgements}
This research was partially supported by NSF grant DMS-0900524.  The
author would like to thank K. Datchev, L. Hillairet, J. Metcalfe, E. Schenck,
M. Taylor, A. Vasy, and
J. Wunsch for many helpful and stimulating discussions.

\section{Preliminaries}

\subsection{The geometry}

We have assumed that our manifold $X$ has a warped product structure with
one or two infinite ends which are short range perturbations of
$\reals^n$.  This means we are considering the manifold $X = \reals_x
\times \Omega^{n-1}_\theta$ (or $X = \reals_x^+ \times
\Omega^{n-1}_\theta$ if one infinite end), equipped with the metric
\[
g = d x^2 + A^2(x) G_\theta,
\]
where $A \in \Ci$ is a smooth function, $A \geq \epsilon>0$ for some
epsilon (or $A(x) = x$ for $x>0$ near $0$ if one infinite end), and $G_\theta$ is the metric on a smooth compact $n-1$ dimensional Riemannian
manifold $\Omega^{n-1}$ without boundary.    The short
range assumption means that as $| x | \to \infty$, we have
\[
| \partial^\alpha ( g - g_E ) | \leq C_\alpha \lll x \rrr^{-2 - |
  \alpha| },
\]
where 
\[
g_E = dx^2 + x^2 G_\theta.
\]
This means that $X$ is asymptotically Euclidean as $| x | \to \infty$.  This
assumption merely allows us to use standard techniques to glue
resolvent estimates together without worrying about trapping at
infinity.  The assumptions can of course be weakened to ``long-range''
perturbation (following Vasy-Zworski \cite{VaZw}), but this paper is
really about the local phenomenon of trapping rather than having the
most general ``infinity''.  We use the notation $\theta \in \Omega^{n-1}$ to denote the
``angular'' directions.  This is in analogue with the case of spherically
symmetric warped product spaces where $\Omega^{n-1} = \SS^{n-1}$ is
the sphere and $X$ is asymptotically $\reals^n$.

From this metric, we get the volume form
\[
d \Vol = A(x)^{n-1} dx d \sigma,
\]
where $\sigma$ is the volume measure on $\Omega^{n-1}$.  The
Laplace-Beltrami operator acting on $0$-forms is computed: 
\[
\Delta f = (\partial_x^2 + A^{-2} \Delta_{\Omega^{n-1}} + (n-1)A^{-1}
A' \partial_x) f,
\]
where $\Delta_{\Omega^{n-1}}$ is the (non-positive) Laplace-Beltrami
operator on $\Omega^{n-1}$.

We want to exploit the warped product structure to reduce spectral questions
to a one-dimensional problem.  Let us first conjugate to a problem on
the flat cylinder.  That is, let $Tu(x, \theta) = A^{(n-1)/2}(x) u(x,
\theta)$ so that $\tDelta = T \Delta T^{-1}$ is essentially
self-adjoint on $L^2(dx d \sigma)$, where $\sigma$ is the usual
volume measure on $\Omega^{n-1}$.  We have
\[
-\tDelta = -\p_x^2 - A^{-2}(x) \Delta_\Omega + V_1(x),
\]
where
\[
V_1(x) = \frac{n-1}{2} A'' A^{-1} - \frac{(n-1)(n-3)}{4} (A')^2 A^{-2}.
\]
Separating variables we write for $u \in L^2(dx d \sigma)$ 
\[
u(x, \theta) = \sum_{l,k} u_{lk}(x) \phi_{lk}(\theta),
\]
where $\phi_{lk}(\theta)$ are the eigenfunctions on $\Omega^{n-1}$ with eigenvalue
$\lambda_k^2$.  Then
\[
-\tDelta u = \sum_{l,k} \phi_{lk}(\theta) Q_k u_{lk},
\]
where
\[
Q_k \phi(x) = ( -\p_x^2 + \lambda_k^2 A^{-2}(x) + V_1(x) ) \phi(x).
\]
Setting $h = \lambda_k^{-1}$ and rescaling, we end up with the
semiclassical operator
\[
P(h) \phi(x) = (-h^2 \p_x^2 + V(x) ) \phi(x),
\]
where
\[
V(x) = A^{-2}(x) + h^2V_1(x).
\]
We sometimes will write $V_0(x) = A^{-2}(x)$ for the principal part of
the effective potential.

The semiclassical versions of Theorem \ref{T:dichotomy} and Corollary
\ref{C:little-o} are given in the following.

\begin{theorem}
\label{T:dichotomy-sc}
Under the assumptions above, either

{\bf 1:}  For every $\phi \in \Ci_c(\reals)$ and every $\epsilon>0$ there exists $C_\epsilon
>0$ such that
\[
\| \phi (-h^2 \p_x^2 + V -(z-i0))^{-1} \phi \| \leq C_\epsilon h^{-2-\epsilon} ,
\,\,\, z \in I,
\]
for a compact interval $I$, or

{\bf 2:} For every $N>0$, there exists $\phi \in \Ci_c(X)$, $C_N>0$,  and $z \in \reals$, $z \neq 0$,  such that 
\[
\| \phi (-h^2\p_x^2 + V - (z-i0))^{-1} \phi \| \geq C_N h^{-N},
\]
along a subsequence as $h \to 0+$.

\end{theorem}

\begin{corollary}
\label{C:little-o-sc}
In addition to the assumptions of Theorem \ref{T:dichotomy-sc}, assume
that the manifold $X$ is analytic.

{\bf 1:}  There exists $\delta>0$ such that, for every $\phi \in
\Ci_c(\reals)$ there is a constant  $C
>0$ for which we have the following estimate
\[
\| \phi (-h^2 \p_x^2 + V -(z-i0))^{-1} \phi \|  \leq C h^{-2+\delta} ,
\,\,\, z \in I,
\]
for a compact interval $I$, or

{\bf 2:} For any $N>0$, there exists $\phi \in \Ci_c(X)$, $C_N>0$, and $z \in \reals$, $z \neq 0$, such that 
\[
\| \phi (-h^2\p_x^2 + V - (z-i0))^{-1} \phi \| \geq C_N h^{-N}
\]
along a subsequence as $h \to 0+$.

\end{corollary}

\subsection{The 0-Gevrey class $\GG^0_\tau$}

Our manifolds already have very nice geometry as $| x | \to \infty$,
and moreover we have separated variables.  Since $\Omega^{n-1}$ is a
$\Ci$ compact manifold without boundary, the only additional
regularity assumptions we need to impose will be at the critical elements
of the manifold $X$, that is, at the critical points of the function $A(x)$.  
In order to have a
meaningful symbol class (especially once we are working with the
calculus of 2 parameters), we need to know that near the critical
elements, the function $A$ is not too far away from being analytic.  For this, we introduce the
following 
0-Gevrey classes of functions with respect to order of vanishing.  For $0 \leq \tau <
\infty$, let $\GG^0_\tau (
\reals )$ be the set of all smooth functions $f : \reals \to \reals$
such that, for each $x_0 \in \reals$, there exists a neighbourhood $U \ni
x_0$ and a constant $C$ such that, for all $0 \leq s \leq k$, 
\[
| \partial_x^k f(x) -\partial_x^k f(x_0) | \leq C (k!)^C | x - x_0
|^{-\tau (k-s) } | \partial_x^s f(x) - \partial_x^s f(x_0) |, \,\,\, x
\to x_0 \text{ in } U.
\]
This definition says that the order of vanishing of derivatives of a
function is only polynomially worse than that of lower derivatives.
Every analytic function is in one of the 0-Gevrey classes $\GG^0_\tau$
for some $\tau < \infty$, but many more functions are as well.   
For example, the function
\[
f(x) = \begin{cases} \exp (-1/x^p), \text{ for } x >0, \\ 0, \text{ for }x \leq 0
\end{cases}
\]
is in $\GG^0_{p+1}$, but
\[
f(x) = \begin{cases} \exp (-\exp(1/x)), \text{ for } x >0, \\ 0, \text{ for }x \leq 0
\end{cases}
\]
is not in any 0-Gevrey class for finite $\tau$.  This implies that the
0-Gevrey class contains a rich subset of functions with compact
support as well as functions which are constant on intervals.

The 0-Gevrey class assumption will only come in to play in the
case of infinitely degenerate critical points (see Subsection \ref{SS:unst-inf}).

\subsection{Semiclassical calculus with $2$ parameters}

Following Sj\"ostrand-Zworski \cite[\S3.3]{SjZw-frac} and \cite{ChWu-lsm}, we introduce a
calculus with two parameters.  We will not present the proofs in the
following lemmas, as they have appeared in several other places, but
merely include the statements for the reader's convenience, as well as
pointers to where proofs can be found.

For $\alpha\in [0,1]$ and $\beta\leq 1-\alpha,$ we let
\begin{eqnarray*}
\lefteqn{\s_{\alpha,\beta}^{k,m, \widetilde{m}} \left(T^*(\RR^n) \right):= } \\
& = & \Bigg\{ a \in \Ci \left(\RR^n \times (\RR^n)^* \times (0,1]^2 \right):  \\
 && \quad \quad  \left| \partial_x^\rho \partial_\xi^\gamma a(x, \xi; h, \tilde{h}) \right| 
\leq C_{\rho \gamma}h^{-m}\tilde{h}^{-\widetilde{m}} \left(
  \frac{\tilde{h}}{h} \right)^{\alpha |\rho| + \beta |\gamma|} 
\langle \xi \rangle^{k - |\gamma|} \Bigg\}.
\end{eqnarray*}
Throughout this work we will always assume $\tilde{h} \geq h$.  
We let $\Psi_{\alpha,\beta}^{k, m, \widetilde{m}}$ denote the
corresponding spaces of semiclassical pseudodifferential operators
obtained by Weyl quantization of these symbols. We will sometimes add a
subscript of $h$ or $\tilde{h}$ to indicate which parameter is used in
the quantization; in the absence of such a parameter, the quantization
is assumed to be in $h.$  The class $\s_{\alpha,\beta}$ (with no
superscripts) will denote $\s_{\alpha,\beta}^{0,0,0}$ for brevity.

In \cite{SjZw-frac} (for the homogeneous case $\alpha=\beta=1/2$), and
in \cite{ChWu-lsm} (for the inhomogeneous case $\alpha \neq \beta$), 
it is observed that the composition in the
calculus can be computed in terms of a symbol product that converges
in the sense that terms improve in $\tilde{h}$ and $\xi$ orders, but
not in $h$ orders.
This happens because when $\alpha + \beta = 1$, the $(h^{-\alpha},
h^{-\beta})$ calculus is marginal, which is what the rescaling (blowup) and
introduction of the second parameter $\tilde{h}$ accomplishes.  
In the sequel, we will always assume we are in the inhomogeneous
marginal case:
$$
\alpha+\beta=1.
$$
If $\alpha + \beta <1$, then of course the calculus is no longer
marginal and computations become much easier.

By the same arguments employed in \cite{SjZw-frac} (see \cite{ChWu-lsm}), we may easily
verify that the calculus $\Psi_{\alpha,\beta}$ is closed under composition: if
$ a \in   \s^{k,m , {\widetilde m} }_{\alpha,\beta}$ and
$ b \in   \s^{k',m', \widetilde m''}_{\alpha,\beta} $ then 
\[ \Op_h^w (a) \circ \Op_h^w(b) = \Op_h^w (c)
\ \text{ with } \ 
c \in   \s^{k+k',m +m', {\widetilde m}+ {\widetilde m}'  }_{\alpha,\beta}\,.
\]
In addition, as in \cite{ChWu-lsm}, we have a symbolic expansion for
$c$ in powers of $ \tilde h $. 

We have the following Lemma from \cite{ChWu-lsm}, which is a more
general version of \cite[Lemma 3.6]{SjZw-frac}:
\begin{lemma}
\label{l:err}
Suppose that 
$ a, b \in \s_{\alpha,\beta}$, 
and that $ c^w = a^w \circ b^w $. 
Then 
\begin{equation}
\label{eq:weylc}  c ( x, \xi) = \sum_{k=0}^N \frac{1}{k!} \left( 
\frac{i h}{2} \sigma ( D_x , D_\xi; D_y , D_\eta) \right)^k a ( x , \xi) 
b( y , \eta) |_{ x = y , \xi = \eta} + e_N ( x, \xi ) \,,
\end{equation}
where for some $ M $
\begin{equation}
\label{eq:new1}
\begin{split}
& | \partial^{\gamma} e_N | \leq C_N h^{N+1}
 \\
& \ \ 
\times \sum_{ \gamma_1 + \gamma_2 = \gamma } 
 \sup_{ 
{{( x, \xi) \in T^* \RR^n }
\atop{ ( y , \eta) \in T^* \RR^n }}} \sup_{
|\rho | \leq M  \,, \rho \in \NN^{4n} }
\left|
\Gamma_{\alpha, \beta, \rho,\gamma }(D)
( \sigma ( D) ) ^{N+1} a ( x , \xi)  
b ( y, \eta ) 
\right| \,,
\end{split} 
\end{equation}
where $ \sigma ( D) = 
 \sigma ( D_x , D_\xi; D_y, D_\eta )  $ as usual,  
and 
\[
\Gamma_{\alpha, \beta, \rho,\gamma }(D) =( h^\alpha \pa_{(x,y)},
h^\beta \pa_{(\xi,\eta)}))^\rho \partial^{\gamma_1} 
\partial^{\gamma_2}.
\]
\end{lemma}

As a particular consequence we notice that if $ a \in 
\s_{\alpha,\beta} ( T^* \RR^n)  $ and $ b \in  \s ( T^* \RR^n ) $ then 
\begin{align}
\label{eq:abc}
  c ( x, \xi) = &
 \sum_{k=0}^N \frac{1}{k!} \left( 
i h \sigma ( D_x , D_\xi; D_y , D_\eta) \right)^k a ( x , \xi) 
b( y , \eta) |_{ x = y , \xi = \eta} \\
& + {\mathcal O}_{\s_{\alpha,\beta}}
\left( h^{N+1} \max \left\{ (\tilde{h}/h)^{(N+1) \alpha}, (\tilde{h}/h)^{(N+1) \beta} \right\} \right)
\,. \notag
\end{align}

We will let $\B$ denote the ``blowdown map'' 
\begin{equation}\label{blowdown}
(x,\xi)=\B(X,\Xi)=((h/\th)^\alpha X, (h/\th)^\beta \Xi).
\end{equation}
The spaces of operators $\Psi_h$ and 
$\Psi_{\tilde{h}}$ are related via a unitary rescaling in the
following fashion.  
Let $a \in \s_{\alpha,\beta}^{k,m,\tilde{m}}$, and consider the
rescaled symbol
\be
a\left(\csh^{\alpha}X, \csh^{\beta} \Xi
\right)= a \circ \B \in \s_{0,0}^{k,m,\tilde{m}}.
\ee
Define the unitary operator $T_{h, \tilde{h}} u(X) = \csh^{\frac{n\alpha}{2}}u\left(
  \csh^{\alpha} X \right)$,   
so that
\be\label{rescaledquantization}
\Op_{\tilde{h}}^w(a\circ B) T_{h, \tilde{h}} u= T_{h, \tilde{h}} \Op_h^w(a) u.
\ee

\section{The trapping}

In order to prove Theorem \ref{T:dichotomy-sc} and Corollary \ref{C:little-o-sc}, we consider the
critical points of the potential $V(x)$, or more specifically the
critical points of the principal part, $V_0(x) = A^{-2}$, of $V = A^{-2} + h^2 V_1$.  The assumption that the
trapped set has only finitely many connected components implies that
the potential $V_0(x)$ has only finitely many critical {\it values}.  We break
the analysis of the critical values into those for which the
Hamiltonian flow of the principal part of our symbol, $p_0 = \xi^2 + V_0(x)$, is locally unstable (either
``genuinely'' unstable or of transmission inflection type), and those
for which the Hamiltonian flow is stable.  This leads to the dichotomy
in Theorem \ref{T:dichotomy-sc} and Corollary \ref{C:little-o-sc}.
The idea is that, if there is a critical value for which the
Hamiltonian flow is stable, then we can immediately construct very
good quasimodes and reach the second conclusions in Theorem
\ref{T:dichotomy-sc} and Corollary \ref{C:little-o-sc}.  This is
relatively straightforward and written in Subsection \ref{SS:stable}.

On the other hand, if there is no stable trapping, then all trapping
is unstable, consisting of 
disjoint critical sets, and even if two critical sets exist at the
same potential energy level, they must be separated by an unstable
maximum 
critical value at a {\it higher} potential energy level (otherwise
there would be a minimum in between, and hence at least weakly stable trapping), so they do
not see each other.  That is to say, the weakly stable/unstable
manifolds of the separating maximum form a separatrix in the reduced
phase space.  This allows us to glue together microlocal
estimates near each critical set, and the resolvent estimate is then
simply the worst of these estimates.  Hence it suffices to classify
microlocal resolvent estimates in a neighbourhood of any of these
unstable critical sets.  This is accomplished in Subsections
\ref{SS:unst-nd}-\ref{SS:inf-deg-infl}.  In this sense, this section
contains a catalogue of microlocal resolvent estimates.

It is important to note at this point that for unstable trapping of
finite degeneracy, the relevant resolvent estimates are all
$o(h^{-2})$, that is to say, the sub-potential $h^2 V_1$ is always of
lower order.  If the trapping is unstable but infinitely degenerate,
we need to work harder to absorb the sub-potential.  The 0-Gevrey
assumption will be important here.

\subsection{Unstable nondegenerate trapping}

\label{SS:unst-nd}

Unstable nondegenerate trapping occurs when the potential $V_0$ has a
nondegenerate maximum.  As mentioned previously, let us for the time
being consider the operator $\tQ = -h^2 \p_x + V_0(x)-z$, where $V_0(x) =
A^{-2}(x)$.  To say that $x = 0$ is a nondegenerate maximum means that
 $x = 0$ is a critical point of
$V_0(x)$ satisfying $V_0'(0) = 0$, $V_0''(0) < 0$, and then the Hamiltonian flow
of $\tq = \xi^2 + V_0(x)$ near $(0,0)$ is
\[
\begin{cases}
\dot{x} = 2 \xi, \\ 
\dot{\xi} = -V_0'(x) \sim x,
\end{cases}
\]
so that the stable/unstable manifolds for the flow are transversal at
the critical point $(0,0)$.

The following result as stated can be read off from \cite{Chr-NC,Chr-NC-erratum,Chr-QMNC}, and has also been
studied in slightly different contexts in \cite{CdVP-I,CdVP-II} and
\cite{BuZw-bb}, amongst many others.  We only pause briefly to remark
that, since the lower bound on the operator $\tQ$ is of the order
$h/\log(1/h) \gg h^2$, the same result applies equally well to $\tQ +
h^2 V_1$.

\begin{proposition}
\label{P:ml-inv-0}
Suppose $x = 0$ is a nondegenerate local maximum of the potential $V_0$,
$V_0(0) = 1$.  For 
$\epsilon>0$ sufficiently small, let $\phi \in \s(T^* \reals)$
have compact support in $\{ |(x,\xi) |\leq \epsilon\}$.  Then there
exists $C_\epsilon>0$ such that 
\begin{equation}
\label{E:ml-inv-0}
\| \tQ \phi^w u \| \geq C_\epsilon \frac{h}{\log(1/h)} \|
\phi^w u \|, \,\,\, z \in [1-\epsilon, 1 + \epsilon].
\end{equation}
\end{proposition}

\subsection{Unstable finitely degenerate trapping}

\label{SS:unst-fin}

In this subsection, we consider an isolated critical point leading to
unstable but finitely degenerate trapping.  That is, we now assume
that $x = 0$ is a degenerate maximum for the function $V_0(x) = A^{-2}(x)$ of order
$m \geq 2$.  If we again assume $V_0(0) = 1$, then this means that near
$x = 0$, $V_0(x) \sim 1 - x^{2m}$.  Critical points of this form were
studied in \cite{ChWu-lsm}, but the proof can also be more or less
deduced from the proofs of Propositions \ref{P:ml-inv-3a} and
\ref{P:ml-inv-3b} below.
We only
remark briefly that again, since the lower bound on the operator $\tQ$
is of the order $h^{2m/(m+1)} \gg h^2$, the estimate applies equally
well to $\tQ + h^2 V_1$.

\begin{proposition}
\label{P:ml-inv-1}
Let $\tQ = -h^2 \p_x^2 + V_0(x) -z$.  For $\epsilon>0$ sufficiently small, let $\phi \in \s(T^* \reals)$
have compact support in $\{ |(x,\xi) |\leq \epsilon\}$.  Then there
exists $C_\epsilon>0$ such that 
\begin{equation}
\label{E:ml-inv-1}
\| \tQ \phi^w u \| \geq C_\epsilon h^{2m/(m+1)} \|
\phi^w u \|, \,\,\, z \in [1-\epsilon, 1 + \epsilon].
\end{equation}
\end{proposition}

\begin{remark}
In \cite{ChWu-lsm}, it is also shown that this estimate is sharp in
the sense that the exponent $2m/(m+1)$ cannot be improved.

\end{remark}

\subsection{Finitely degenerate inflection transmission trapping}

We next study the case when the potential has an inflection point of
finitely degenerate type.  That is, let us assume the point $x = 1$ is
a finitely degenerate inflection point, so that locally near $x = 1$,
the potential $V_0(x) = A^{-2}(x)$ takes the form
\[
V_0(x) \sim C_1^{-1} -c_2(x-1)^{2m_2 + 1}, \,\, m_2 \geq 1
\]
where $C_1>1$ and $c_2>0$.  Of course the constants are arbitrary
(chosen to agree with those in \cite{ChMe-lsm}), and $c_2$ could be
negative without changing much of the analysis.  This Proposition and the
proof are in \cite{ChMe-lsm}, and as we will once again
revisit the proof of this Proposition in Subsection \ref{SS:inf-deg-infl},
we will omit it at this point.  But one last time, let us observe that
since the lower bound on the operator $\tQ$ is of the order $h^{(4m_2 +
2)/(2m_2+3)} \gg h^2$, the estimate applies equally well to the
operator $\tQ + h^2 V_1$.

\begin{proposition}
\label{P:ml-inv-2}
For $\epsilon>0$ sufficiently small, let $\phi \in \s(T^* \reals)$
have compact support in $\{ |(x-1,\xi) |\leq \epsilon\}$.  Then there
exists $C_\epsilon>0$ such that 
\begin{equation}
\label{E:ml-inv-2}
\| \tQ \phi^w u \| \geq C_\epsilon
h^{(4m_2+2)/(2m_2+3)}  \|
\phi^w u \|, \,\,\, z \in [C_1^{-1}-\epsilon, C_1^{-1} + \epsilon].
\end{equation}
\end{proposition}

\begin{remark}
We remark that in this case, \cite{ChMe-lsm} shows once again that
this estimate is sharp in the sense that the exponent
$(4m_2+2)/(2m_2+3)$ cannot be improved.

\end{remark}

\subsection{Unstable infinitely degenerate and cylindrical trapping}
\label{SS:unst-inf}

In this subsection, we study the case where the principal part of the
potential $V(x) = A^{-2}(x) + h^2 V_1(x)$ has
an infinitely degenerate maximum, say, at the point $x = 0$.  Let
$V_0(x) = A^{-2}(x)$.  As
usual, we again assume that $V_0(0) = 1$, so that
\[
V_0(x) = 1 - \O(x^\infty)
\]
in a neighbourhood of $x = 0$.  Of course this is not very precise, as
$V_0$ could be constant in a neighbourhood of $x = 0$ and still satisfy
this, and the proof must be modified to suit these two cases.  So let us first assume that $V_0(0) = 1$, and
$V_0'(x)$ vanishes to infinite order at $x = 0$, however, $\pm V_0'(x) <0$
for $\pm x >0$.  That is, the critical point at $x = 0$ is infinitely
degenerate but isolated.

Our microlocal spectral theory result is then that the microlocal
cutoff resolvent is bounded by $\O_\eta (h^{-2-\eta})$ for any $\eta>0$.  In order to state the
result,  let 
\[
\tQ = -h^2 \p_x^2 + V(x) -z = -h^2 \p_x^2 + A^{-2}(x) + h^2 V_1(x) -z.
\]

\begin{proposition}
\label{P:ml-inv-3a}
For $\epsilon>0$ sufficiently small, let $\phi \in \s(T^* \reals)$
have compact support in $\{ |(x,\xi) |\leq \epsilon\}$.  Then for any
$\eta>0$, there
exists $C_{\epsilon, \eta}>0$ such that 
\begin{equation}
\label{E:ml-inv-3a}
\| \tQ \phi^w u \| \geq C_{\epsilon, \eta} h^{2+ \eta} \|
\phi^w u \|, \,\,\, z \in [1-\epsilon, 1 + \epsilon].
\end{equation}
\end{proposition}

\begin{remark}
As this is the limiting case as $m \to \infty$ of Proposition
\ref{P:ml-inv-1}, we believe the optimal lower bound in this case is $h^2/
\gamma(h)$ for some $\gamma(h) \to 0$.  This is further suggested by a
microlocal scaling
heuristic.  However, various attempts to
tighten up the argument to get the better lower bound seem to fail.
It would be very interesting to determine if a lower bound of $h^2 /
\gamma(h)$ or even $h^2$ holds.

\end{remark}

For our next result, we consider the case where there is a whole
cylinder of unstable trapping.  That is, we assume the principal part
of the effective potential
$V_0(x)$ has a maximum $V_0(x) \equiv 1$ on an interval, say $x \in
[-a,a]$, and that $\pm V_0'(x) <0$ for $\pm x > a$.  Our main result in
this case says that the microlocal cutoff resolvent is again 
controlled by $h^{-2- \eta}$ for any $\eta>0$.  Let us again set 
\[
\tQ = -h^2 \p_x^2 + V(x) -z.
\]

\begin{proposition}
\label{P:ml-inv-3b}
For $\epsilon>0$ sufficiently small, let $\phi \in \s(T^* \reals)$
have compact support in $\{ |x| \leq a + \epsilon,\, |\xi |\leq
\epsilon\}$.  Then for any $\eta>0$, there
exists $C_{\epsilon, \eta}>0$ such that 
\begin{equation}
\label{E:ml-inv-3b}
\| \tQ \phi^w u \| \geq C_{\epsilon, \eta} h^{2+ \eta}\|
\phi^w u \|, \,\,\, z \in [1-\epsilon, 1 + \epsilon].
\end{equation}
\end{proposition}

\begin{remark}
For similar reasons, we expect the optimal lower bound in this case
should be $h^2$.
\end{remark}

\begin{proof}
The proof of these Propositions is very similar, so we put them together.
We will first prove Proposition \ref{P:ml-inv-3a}, and then point out how
the proof must be modified to get Proposition \ref{P:ml-inv-3b}.

The idea of the proof of Proposition \ref{P:ml-inv-3a} (and indeed Proposition
\ref{P:ml-inv-3b}) is to add a small $h$-dependent bump with a {\it
  finitely degenerate} maximum, and then use the result of Proposition
\ref{P:ml-inv-1}.  Of course the bump has to be sufficiently small
that the operator $\tQ$ is close to the perturbed operator.

Choose a point $x_0 = x_0(h) >0$ and $\epsilon>0$ so that $x_0$ is
the smallest point such that 
\[
- x V_0'( x ) \geq \frac{ h}{\varpi(h)}, \,\, \, x_0 \leq | x | \leq \epsilon,
\]
where $\varpi(h)$ will be determined later.  As long as $\varpi(h) \gg
h$, this implies that $x_0 = o(1)$.  We remark that, of course, $x_0$
depends also on the choice of $\varpi(h)$, but for any $\varpi$, there
is such a choice, since $V_0'$ vanishes to infinite order at $x = 0$.  Further, as
$V_0'(x) = \O(x^\infty)$ near $x = 0$, we have $x_0 \gg h^\delta$ for
any $\delta>0$.  Fix $m \geq
2$ to be determined later in the proof ($m$ will be large), and choose also
an even function $f \in \Ci_c([-2,2]) \cap \GG^0_\tau$ for some $\tau
< \infty$, with $f(x) = 1 - \frac{1}{2m}
x^{2m}$ for $| x | \leq 1$, and $f' (x) \leq 0$ or $x \in [0,2]$.  For
another parameter $\Gamma(h)>0$ to be determined, let
\[
W_h(x) = \Gamma(h) f(x/x_0),
\]
and let
\[
V_{0,h}(x) = V_0(x) + W_h(x)
\]
and
\[
V_h(x) = V(x) + W_h(x)
\]
(see Figure \ref{fig:Vh-1}).  The parameter $\Gamma(h)$ will be seen
to be $h^{2+ \eta}$ for $\eta >0,$ $\eta = \O(m^{-1})$ as $m \to \infty$.   
\begin{figure}
\hfill
\centerline{\input{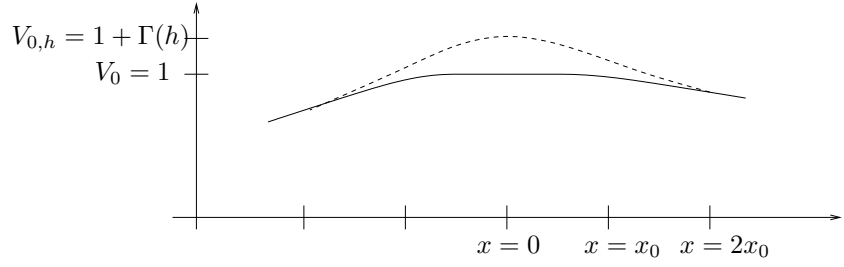}}
\caption{\label{fig:Vh-1} The potential $V_0$ and the modified potential
  $V_{0,h}$ (in dashed). }
\hfill
\end{figure}
By construction,
\[
| V(x) - V_h(x) | \leq | W_h | \leq \Gamma(h).
\]

Let $Q_1 = (hD)^2 + V_h$ with symbol $q_1 = \xi^2 + V_h$.  
The Hamilton vector field $\hamvf$ associated to the symbol $q_1$ is given by
\begin{align*}
\hamvf & = 2\xi\pa_{x} -V_h' \pa_{\xi} \\
& = 2 \xi \pa_x - \left( \frac{\Gamma(h)}{x_0} f'(x/x_0) + V_0'(x) +
  h^2 V_1'(x)
\right) \p_\xi .
\end{align*}

We will use the same change of coordinates and the same singular
commutant as in \cite{ChWu-lsm}, but we also have to track the loss
coming from the coefficient $\Gamma(h)$.  For $\alpha = 1/(m+1)$, let 
$$
\Xi=\frac{\xi}{(h/\th)^{m\alpha}},\quad X = \frac{x}{(h/\th)^\alpha},
$$
so that in the new blown-up coordinates $\Xi,X,$
\begin{equation}\label{blownupvf-3a}
\hamvf= (h/\th)^{\frac{m-1}{m+1}}\big(\Xi \pa_X -
(h/\th)^{(1-2m)/(m+1)} V_h ' (
(h/\th)^\alpha X) \pa_\Xi \big)
\end{equation}

Let $\Lambda(s)$ be defined as in \cite{ChWu-lsm} by fixing
$\epsilon_0>0$ and setting 
\[
\Lambda(s) = \int_0^s \lll z \rrr^{-1-\ep_0} dz,
\]
so that $\Lambda$ is a zero order symbol satisfying $\Lambda(s) \sim
s$ for $s$ near $0$.  
Following \cite{Chr-NC,Chr-QMNC,ChWu-lsm}, we define 
$$
a(x,\xi;h) = \Lambda(\Xi)\Lambda(X)\chi(x)\chi(\xi)= \Lambda(\xi/(h/\th)^{m\alpha}) \Lambda(x/(h/\th)^\alpha)\chi(x)\chi(\xi),
$$
where $\chi(s)$ is a cutoff function equal to $1$ for $\abs{s}<\delta_1$
and $0$ for $s>2\delta_1$ ($\delta_1$ will be chosen shortly).
Then $a$ is bounded, and a $0$ symbol in $X,\Xi:$
$$
\abs{\pa_X^\alpha \pa_\Xi^\beta a}\leq C_{\alpha,\beta}.
$$
(Recall that $x=(h/\th)^\alpha X$ and $\xi=(h/\th)^{m\alpha}\Xi.$)
Using \eqref{blownupvf-3a}, it is simple to
compute
\begin{equation}\label{gdefn-3a}
\begin{aligned}
\hamvf (a) = & (h/\th)^{\frac{m-1}{m+1}}\chi(x)\chi(\xi)\big(
\Lambda(\Xi)
\ang{X}^{-1-\ep_0}\Xi \\
& - (h/\th)^{(1-2m)/(m+1)} \Lambda(X) V_h'( (h/\th)^\alpha X)
\ang{\Xi}^{-1-\ep_0} \big)+r\\
\equiv & (h/\th)^{\frac{m-1}{m+1}} g+r
\end{aligned}
\end{equation}
with $$\supp r\subset \{\abs{x}>\delta_1\} \cup \{\abs{\xi}>\delta_1\}$$
($r$ comes from terms involving derivatives of $\chi(x)\chi(\xi)$).

For $| X |
\leq (h/\th)^{-\alpha} x_0$, we have
\begin{align*}
- (h/\th)^{(1-2m)/(m+1)} & \Lambda(X) V_h'( (h/\th)^\alpha X)
\ang{\Xi}^{-1-\ep_0}  \\
& =  \Gamma(h) x_0^{-2m} \Lambda(X) X^{2m-1} 
\ang{\Xi}^{-1-\ep_0} + g_2,
\end{align*}
with 
\[
g_2 = -(h/\th) ^{(1-2m)/(m+1)} \Lambda(X) (V_0'+h^2V_1')( (h/\th)^\alpha X)
\ang{\Xi}^{-1-\ep_0}  .  
\]
Note that we always have $-\Lambda(x) V_0'(x) \geq 0$, so we expect
the quantization of $g_2$ to be at least semibounded below.  This is
demonstrated in Lemma \ref{L:g2g3} below.

 For $|X| \leq (h/\th)^{-\alpha} x_0$ and $|
\Xi | \leq (h / \th )^{-\alpha m} \delta_1$ consider 
\begin{align*}
g = & \chi(x)\chi(\xi)\big(
\Lambda(\Xi)
\ang{X}^{-1-\ep_0}\Xi \\
& - (h/\th)^{(1-2m)/(m+1)} \Lambda(X) V_h'( (h/\th)^\alpha X)
\ang{\Xi}^{-1-\ep_0} \\
= & \Lambda(\Xi) \Xi \ang{X}^{-1-\ep_0} + \Gamma(h) x_0^{-2m} \Lambda(X) X^{2m-1} 
\ang{\Xi}^{-1-\ep_0} + g_2 \\
= & \lambda^2 \left( \lambda^{-1}\Lambda ( \Xi ) (\lambda^{-1}\Xi )
  \ang{X}^{-1-\ep_0} +  \lambda^{-2m-2}\Gamma x_0^{-2m} \lambda \Lambda (X) ( \lambda X)^{2m-1} \ang{\Xi}^{-1-\ep_0} \right)+ g_2 \\
= & \lambda^2 \left( \tLambda_1( \Xi') \Xi' \ang{ \lambda^{-1} X'
  }^{-1-\ep_0} + \lambda^{-2m-2}\Gamma x_0^{-2m} \tLambda_2 (X') (X')^{2m-1} \ang{ \lambda \Xi'
  }^{-1-\ep_0} \right) + g_2 \\
=: & g_1 + g_2,
\end{align*}
where we have used the $L^2$-unitary rescaling
\[
X' = \lambda X, \,\,\, \Xi' = \lambda^{-1}\Xi,
\]
and $\lambda>0$ (small) will be determined in the course of the proof.

The functions $\tLambda_j$, $j = 1,2$, are defined by changing variables:
\[
\tLambda_1 ( \Xi' ) = \lambda^{-1} \Lambda ( \Xi) = \lambda^{-1}
\Lambda ( \lambda \Xi' ),
\]
and
\[
\tLambda_2( X') = \lambda \Lambda (X) = \lambda \Lambda ( \lambda^{-1}
X' ).
\]
The error term $g_2$ is the term in the expansion of $g$ coming
from estimating using $W_h'$ rather than $V_h'$.  We will deal with $g_2$ in due course.  
We are now microlocalized on a set where
\[
| X' | \leq \lambda (h/\th)^{-\alpha} x_0, \,\,\, | \Xi' | \leq \lambda^{-1} (h/\th)^{-m\alpha} \delta_1,
\]
and will be quantizing in the $\th$-Weyl calculus, so we need symbolic
estimates on these sets.  

If 
\[
| \lambda^{-1} X'| \leq  \delta_1, \text{ and } | \lambda \Xi' | \leq \delta_1,
\]
and $\delta_1>0$ is sufficiently small, 
then $\tLambda_1( \Xi') \sim \Xi'$ and $\tLambda_2( X' ) \sim X'$, so
that $g_1$ is bounded below as follows:  
\begin{equation}
\label{E:g1-1}
g_1 \geq \min \left\{ \lambda^2, \lambda^{-2m} \Gamma x_0^{-2m}
\right\} ( (\Xi')^2 + (X')^{2m} ).
\end{equation}
Then the $\th$-quantization of $g_1$ is bounded below microlocally on
this set by this minimum times $\th^{2m/(m+1)}$ (see \cite[Lemma A.2]{ChWu-lsm}).

Now on the complementary set, we have one of either $|
\lambda^{-1} X'|^{1 + \epsilon_0}$ or $| \lambda \Xi' |$ is larger than, say,
$(\delta_1/2)^{1 + \epsilon_0}$.  We also need to keep track of the {\it relative} size
of these two quantities.  
If $\abs{\lambda \Xi'}\geq
 \max\left( \abs{ \lambda^{-1} X'}^{1+\ep_0}, (\delta_1/2)^{1 +
     \epsilon_0} \right)$
then
\begin{align}
g_1 & \geq \lambda^2 \tLambda_1( \Xi') \Xi' \ang{ \lambda^{-1} X'
}^{-1-\ep_0} \notag \\
& \geq c\lambda^2 \tLambda_1( \Xi') \frac{ \Xi'}{ \abs{\lambda \Xi'}}
\notag \\
& = c\lambda \tLambda_1 (\Xi') \sgn (\Xi') \notag \\
& = c \Lambda ( \lambda \Xi'  ) \sgn ( \Xi' ) \notag \\
& \geq c_{\delta_1}. \label{E:g1-2}
\end{align}
Hence the $\th$-quantization of $g_1$ is bounded below by a positive
constant, independent of $h$ and $\th$ on this set.

The remaining set is a bit more difficult.  If 
\[
\abs{\lambda^{-1}  X'}^{1+\ep_0}\geq
\max\left(\abs{\lambda \Xi'}, (\delta_1 /2)^{1 + \ep_0} \right),
\]
 then
\begin{align}
g_1 & \geq \lambda^2 \Big( \tLambda_1 ( \Xi' ) \Xi' \lll
  (h/\th)^{-\alpha} x_0 \rrr^{-1-\epsilon_0} \notag \\
& \quad + \lambda^{-2m-2} \Gamma
  x_0^{-2m} \tLambda_2(X')(X')^{2m-1} \lll \lambda \Xi'  \rrr^{-(1 +
    \epsilon_0)} \Big) \notag \\
& =  \lambda^2 \Big( \lambda^{-1}\Lambda ( \lambda\Xi' ) \Xi' \lll
  (h/\th)^{-\alpha} x_0 \rrr^{-1-\epsilon_0} \notag \\
& \quad + \lambda^{-2m-2} \Gamma
  x_0^{-2m} \lambda \Lambda(\lambda^{-1} X')(X')^{2m-1}       \lll \lambda \Xi' \rrr^{-(1 +
    \epsilon_0)}    \Big) \notag \\
& =  \lambda^2 \Big( \lambda^{-1}\Lambda ( \lambda\Xi' ) \Xi' \lll
  (h/\th)^{-\alpha} x_0 \rrr^{-1-\epsilon_0} \notag \\
& \quad + \lambda^{-2m} \Gamma
  x_0^{-2m}  \Lambda(\lambda^{-1} X')(X')^{2m-2}   (
  \lambda^{-1} X' )       \lll \lambda \Xi'  \rrr^{-(1 +
    \epsilon_0)}      \Big) \notag \\
& \geq \begin{cases}
c_{\delta_1} \min \{ \lambda^2 (h/\th)^{\alpha(1 + \epsilon_0)}
x_0^{-1-\epsilon_0}, \lambda^{-2m+2} \Gamma x_0^{-2m}
 \} \\
\quad \times ( (\Xi')^2 +
(X')^{2m-2}), \text{ if } | \lambda \Xi' | \leq (\delta_1/2)^{1 + \epsilon_0}, \\
c_{\delta_1}'  (h/\th)^{\alpha ( 1 + \epsilon_0)} x_0^{-1-\epsilon_0},
\text{ if } | \lambda \Xi' | \geq (\delta_1/2)^{1 + \epsilon_0}.
\end{cases} \label{E:g1-3}
\end{align}

We now optimize the minimum in \eqref{E:g1-3} to determine $\lambda$
in terms of the other parameters:
\[
\lambda^2 (h/\th)^{\alpha(1 + \epsilon_0)}
x_0^{-1-\epsilon_0}= \lambda^{-2m+2} \Gamma x_0^{-2m},
\]
or
\[
\lambda^2 = \Gamma^{1/m}  (h/\th)^{ -\alpha(1+\epsilon_0 ) / m }
x_0^{ -2 +  (1+ \epsilon_0 )/ m }.
\]
Then the minimum is
\[
\lambda^2 (h/\th)^{ \alpha (1 + \epsilon_0)} x_0^{-1-\epsilon_0} =
\Gamma^{1/m} (h/\th)^{  \alpha(1+\epsilon_0 ) ( m-1)/m  } x_0^{  -3
  -\epsilon_0 + (1+ \epsilon_0 )/m  },
\]
and the $\th$-quantization of $g_1$ on this set is bounded below
microlocally by this number times $\th^{2(m-1)/m}$ (see \cite[Lemma
A.2]{ChWu-lsm}).

\begin{remark}
We pause to remark that here is one place where alternative methods to
optimize the lower bounds give worse results.  For example, on the set
where 
\[
| \lambda \Xi'| \leq (\delta_1/2)^{1+\epsilon_0} \leq | \lambda^{-1}
X'|^{1 + \epsilon_0}, 
\]
we could
estimate $g_1$ from below using only the second term.  This gives a
lower bound of $c_{\delta_1}''\Gamma x_0^{-2m}$, which is much worse
than that computed above.

\end{remark}

Finally, recalling that eventually $\th>0$ will be fixed and $h \ll
\th$, taking the worst lower bound from \eqref{E:g1-1} through
\eqref{E:g1-3}, we obtain for a function $u$ with $h$-wavefront set
contained in the set where $| \lambda^{-1} X'| \leq (h/\th)^{-\alpha}
x_0$, $| \lambda \Xi'| \leq (h/\th)^{-m\alpha} \delta_1$,
\[
\lll \Op_{\th} (g_1) u, u \rrr \geq  \Gamma^{1/m} (h/\th)^{  \alpha(1+\epsilon_0 ) ( m-1)/m  } x_0^{  -3
  -\epsilon_0 + (1+ \epsilon_0 )/m  } \th^{2(m-1)/m} \| u \|^2.
\]

On the other hand, if $ (h/\th)^{-\alpha} x_0 \leq | X | \leq  (h/\th)^{-\alpha} \delta_1$, we have
\begin{align}
- & (h/\th)^{(1-2m)/(m+1)} \Lambda(X) V_h'( (h/\th)^\alpha X)
\ang{\Xi}^{-1-\ep_0} \notag \\
& = - (h/\th)^{(1-2m)/(m+1)} \sgn(X) B(X)
\frac{| (h/\th)^\alpha X|}{| (h/\th)^\alpha X|} V'( (h/\th)^\alpha X)
\ang{\Xi}^{-1-\ep_0} + g_3 \notag \\
& \geq (h/\th)^{(1-2m)/(m+1)} \left( \frac{B(X)}{ | (h/\th)^\alpha X| }
 \frac{h}{\varpi(h)} - \O(h^2) \right) \ang{\Xi}^{-1-\ep_0} + g_3
\notag \\
& \geq C \frac{ h^{(2-m)/(m+1)} \th^{(2m-1)/(m+1)} }{ \varpi(h) }
\ang{\Xi}^{-1-\ep_0} + g_3 \notag \\
& \geq C \frac{ h^{(2-m)/(m+1)} \th^{(2m-1)/(m+1)} }{ \varpi(h) }
(h/\th)^{(1 + \ep_0) m/(1 + m) } + g_3 \notag \\
& \geq 2 \frac{ h^{3/(m+1)} \th^{(m-1)/(m+1)}}{\varpi(h) } + g_3
, \label{E:tV-controls}
\end{align}
where $B(X) \geq c_0 >0$.  
The second inequality holds provided $h /\varpi  \gg h^2$ (so
that $h^2 V_1'$ is controlled by $V_0'$), and the last inequality holds
as $h \to 0$
provided $\ep_0 < 1/m$.  The error $g_3 \geq 0$ comes from using $V'$
in the expansion of $g$ rather than $W_h'$.  

We now deal with the (nearly) positive error terms $g_2$ and $g_3$.  
\begin{lemma}
\label{L:g2g3}
The error terms $g_2$ and $g_3$ are semi-bounded below in the
following sense: if $u(X)$ has wavefront set localized in
\[
\{ |X| \leq \epsilon (h/\th)^{-1/(m+1)}, \,\, | \Xi | \leq \epsilon
(h/\th)^{-m/(m+1)} \},
\]
then for any $\delta >0$ and $N >0$, 
\[
\lll \Op_{\th} ( g_j ) u, u \rrr \geq -C_N h^{(N-2)m/(m+1) - \delta}
\th^{2m/(m+1)} \| u \|^2,
\]
for $j = 2, 3$.

\end{lemma}

\begin{proof}

We prove the relevant bounds for $x \geq 0$.  The analysis for $x \leq
0$ is similar.  For $g_2$, for $N>0$ large, and $\delta>0$ small, choose $0 < x_1 <
x_2 = o(1)$ satisfying
\[
-x_1 V_0'(x_1) = h^{Nm/(m+1)}
\]
and
\[
-x_2 V_0'(x_2) = h^{Nm/(m+1)-\delta}.
\]
As usual, since $V_0'(x) = \O( x^\infty )$, the points $x_j$, $j = 1, 2$
satisfy $x_j \gg h^{\delta_2}$ for any $\delta_2 >0$.  The 0-Gevrey
condition also implies $|x_2-x_1| \gg h^{\delta_2}$ for any
$\delta_2>0$ as well.  To see this, Taylor's theorem says
\[
(V_0'(x_2) - V_0'(x_1) ) = V_0''(\xi) (x_2 - x_1)
\]
for some $x_1 \leq \xi \leq x_2$.  The 0-Gevrey condition and
monotonicity near $x = 0$ implies
\[
| V_0''(\xi) | \leq | V_0''(x_2) | \leq C \left| \frac{V_0'(x_2)}{x_2^\tau} \right|,
\]
so that
\[
(V_0'(x_2) - V_0'(x_1) ) \leq C \left| \frac{V_0'(x_2)}{x_2^\tau} \right| (x_2 -
x_1)
\]
for some $\tau < \infty$, which in turn implies (recalling $V_0' <0$ for $x>0$ near $0$)
\[
(x_2 - x_1 ) \geq C' \frac{ V_0'(x_1) - V_0'(x_2)}{|V_0'(x_2)|} x_2^\tau =
x_2^\tau \left(1 - \left| \frac{V_0'(x_1)}{V_0'(x_2)} \right| \right).
\]
We claim 
\[
\left| \frac{V_0'(x_1)}{V_0'(x_2)} \right| = o(1),
\]
which will finish the proof that $| x_2 - x_1 | \gg h^{\delta_2}$ for
any $\delta_2>0$.  For this, we write
\begin{align*}
\left| \frac{V_0'(x_1)}{V_0'(x_2)} \right| & = \left| x_1
  \frac{V_0'(x_1)}{x_2 V_0'(x_2)} \right| \frac{x_2}{x_1} \\
& = h^\delta \frac{x_2}{x_1}.
\end{align*}
Writing $x_2 = x_1 + \gamma(h)$, we are trying to show $\gamma(h) \gg
h^{\delta_2}$ for any $\delta_2>0$.  For a fixed $\delta_2$, 
if $\gamma(h) \gg h^{\delta_2}$ we're done.  
If
$\gamma(h) < h^{\delta_2}$, then we will produce a contradiction (in
fact showing that $\gamma(h) \gg h^{\delta_2}$).  If $\gamma(h) <
h^{\delta_2}$ for this $\delta_2$, then 
\[
\frac{\gamma(h)}{x_1} \ll 1,
\]
since $x_1 \gg h^{\delta_2}$.  Then it follows that 
\[
h^\delta \frac{x_2}{x_1} = h^\delta \frac{ x_1 + \gamma(h)}{x_1} \ll 2
h^\delta = o(1).
\]
Plugging into our earlier computation, we get
\[
x_2 - x_1 = C' x_2^\tau(1 - o(1)) \gg C'_{\delta_3} h^{\tau \delta_3}
\]
for any $\delta_3>0$.  Taking $\delta_3>0$ sufficiently small so that
$h^{\tau \delta_3} \gg h^{\delta_2}$ implies
\[
\gamma(h) = x_2 - x_1 \gg h^{\delta_2},
\]
which is a contradiction to our assumption that $\gamma(h) \leq h^{\delta_2}$.

Now let $\psi(x)$ be a smooth function, $\psi \geq 0$, $\psi(x) \equiv 1$ on
$[0,x_1]$ with $\psi(x) \equiv 0$ for $x \geq x_2$.  Assume also that
$| \partial_x^k \psi | \leq C_k| x_2 - x_1 |^{-k} = o(h^{-k \delta_2} )$
for any $\delta_2 >0$.  Let $\tpsi(X) = \psi((h/\th)^\alpha X )$,
$\alpha = 1/(m+1)$, so
that
\[
| \p_X^k \tpsi | \leq C_k (h/\th)^{\alpha k} \cdot o(h^{-k \delta_2}).
\]  
We have
\begin{align*}
& \lll \Op_{\th} ( g_2 ) u, u \rrr \\
& \quad = \lll \Op_{\th} ( g_2 ) (1 - \tpsi)
u, (1 - \tpsi ) u \rrr + 
\lll \Op_{\th} ( g_2 ) \tpsi u, \tpsi u \rrr  \\
& \quad \quad + 2\lll \Op_{\th} ( g_2 )
\tpsi u, (1 - \tpsi) u \rrr .
\end{align*}
We estimate each term separately.

On the support of $1-\tpsi$ (again recalling we are only looking at $x
\geq 0$), we have $(h/\th)^{1/(m+1)} X \geq x_1$ so
that in this region we can apply the 0-Gevrey condition to $V_1'$ to
absorb $h^2 V_1'$ into $V_0'$.
Recall that $V_1$ consists of quotients of derivatives of $A$ with
powers of $A$.  The function $A$ is bounded above and below by a
(positive) 
constant for $x$ small, so we are really only concerned with
estimating a finite number of derivatives of $A$.  
Then according to the 0-Gevrey condition, for any $\delta_2 >0$, we have for some $s, \tau < \infty$
\begin{align*}
h^2 | V_1'((h/\th)^\alpha X ) | & \leq C h^2 | x_1 |^{-s \tau}
|A'((h/\th)^\alpha X)| \\
& \leq C h^{2-s \tau \delta_2} | V_0' ((h/\th)^\alpha X ) |,
\end{align*}
and similarly for a finite number of derivatives of $V_1$.  By taking
$\delta_2>0$ sufficiently small we see that on the support of $1 -
\tpsi$, the quantization of $V_0'$ controls that of $h^2 V_1'$.  That
is, for $h>0$ sufficiently small, 
\begin{align*}
& \lll \Op_{\th} ( g_2 ) (1 - \tpsi)
u, (1 - \tpsi ) u \rrr \\
& \quad \geq -C\left(\frac{h}{\th}\right)^{\frac{(1-2m)}{(m+1)}} \lll \Op_{\th} ( \Lambda(X) V_0'( (h/\th)^\alpha X)
\ang{\Xi}^{-1-\ep_0}  ) (1 - \tpsi)
u, (1 - \tpsi ) u \rrr  .
\end{align*}

Then
we calculate in this region 
\begin{align*}
  & \left( \frac{h}{\th} \right)^{(1-2m)/(m+1)}  \left( \frac{
    \Lambda(X)}{(h/\th)^{1/(m+1)} X } \right) \\
& \quad \quad \times \left ( -
  (h/\th)^{1/(m+1)} X V_0'((h/\th)^{1/(m+1)} X ) \right) \lll \Xi
\rrr^{-1-\ep_0} \\
& \quad = h^{(1-2m)/(m+1)} \th^{(2m-1)/(m+1)}  h^{Nm/(m+1)} A(X, h,
\th) \lll \Xi \rrr^{-1-\ep_0}
\end{align*}
where $A$ is a symbol bounded below by a positive constant.  This
follows since
\begin{align*}
X & \geq \left( \frac{h}{\th} \right)^{-\alpha} x_1 \\
& \geq \left( \frac{h}{\th} \right)^{-\alpha} h^{\delta_2}
\end{align*}
for any $\delta_2>0$.  Taking $\delta_2< \alpha$, this lower bound is
(at least)
a positive constant.

On the
set where $A \lll \Xi \rrr^{-1-\ep_0} \geq 1$, this operator is
bounded below, while on the complement, we use the Sharp
G\r{a}rding inequality to get for any $\delta_2 >0$
\begin{align*}
& \lll \Op_{\th} ( g_2 ) (1 - \tpsi)
u, (1 - \tpsi ) u \rrr \\
& \quad \geq -C_{\delta_2} \th h^{((N-2)m + 2)/(m+1) -
  2\delta_2} \th^{(2m-2)/(m+1)} \| (1 - \tpsi)u \|^2.
\end{align*}
For the remaining two terms, on the support of $\tpsi$, we have
$0 \leq (h/\th)^{1/(m+1)} X \leq x_2$.  We know that $| \p_x^k A |$ is
an increasing function for small $x$, so that to estimate $V_1'$, we
estimate a finite number of derivatives of $A$ from above, we can
estimate at the right-hand endpoint $x_2$.  That is, we have as above
for $s, \tau < \infty$ and any $\delta_2>0$,  
\begin{align*}
h^2 | V_1'((h/\th)^\alpha X ) | & \leq C h^2 |x_2|^{-s \tau} |
V_0'(x_2) | \\
& \leq C h^{2 } |x_2|^{-s \tau -1} | x_2 
V_0'(x_2) | \\
& \leq C h^{2-\delta_2(1 + s \tau)} h^{N m/(m+1) - \delta}
\end{align*}
by our choice of $x_2$.  This implies that on the support of $\tpsi$,
$h^2 V_1'$ is controlled by a large power of $h$, by taking $\delta_2>0$
sufficiently small.  
That is,  in this region
\begin{align*}
g_2 & = (h/\th)^{(1-2m)/(m+1)} \Bigg[ \left( \frac{
    \Lambda(X)}{(h/\th)^{\alpha)} X } \right) \left( -
  (h/\th)^{\alpha} X V_0'((h/\th)^{\alpha} X ) \right) \\
& \quad -h^2 \Lambda(X)
V_1'((h/\th)^\alpha X)\Bigg] \lll \Xi
\rrr^{-1-\ep_0} \\
& =  h^{-2m/(m+1)} \th^{(2m)/(m+1)}  h^{Nm/(m+1) - \delta} A_1(X, h,
\th) \lll \Xi \rrr^{-1-\ep_0},
\end{align*}
where $A_1$ is a function satisfying
\[
| \partial_X^k A_1 | \leq C_{k, \delta_2} (h^{1/(m+1) - \delta_2 } \th^{-1/(m+1)}
)^k.
\]
Hence if $\delta_2 < 1/(m+1)$, 
\[
\lll \Op_{\th} ( g_2 ) \tpsi u, \tpsi u \rrr  = \O(h^{(N-2)m/(m+1) -
  \delta} \th^{2m/(m+1)} ) \| u \|^2,
\]
and similarly
\[
\lll \Op_{\th} ( g_2 )
\tpsi u, (1-\tpsi) u \rrr = \O(h^{(N-2)m/(m+1) -
  \delta} \th^{2m/(m+1)} ) \| u \|^2.
\]

The proof for $g_3$ is the same (since we have assumed $f \in \Ci_c
\cap \GG^0_\tau$ for some $\tau < \infty$), but slightly easier, since $g_3$ is
the error term coming from $W_h$ away from $x = 0$, and $W_h$ is
already $\O(h^2)$.

\end{proof}

Let us recap what we have shown so far and fix some of the
parameters.  We have perturbed our potential by a term of size
$\Gamma$, which we want to be much smaller than our lower bound on 
$h \Op_h(\hamvf (a))$.  That is, we want to solve
\begin{align*}
h \left( \frac{h}{\th} \right)^{(m-1)/(m+1)} \Gamma^{1/m} (h/\th)^{  \alpha(1+\epsilon_0 ) ( m-1)/m  } x_0^{  -3
  -\epsilon_0 + (1+ \epsilon_0 )/m  } \th^{2(m-1)/m} 
\gg \Gamma.
\end{align*}
As $m$ will be large, $\epsilon_0 < 1/m$, and $x_0 = o(1)$, it
suffices to solve
\begin{align*}
h^{\frac{2m}{m+1} + \frac{(m-1)(1 + \epsilon_0)}{m(m+1)} } & \th^{-\frac{(m-1)}{(m+1)}
  +2\frac{(m-1)}{m} - \frac{(m-1)(1 + \epsilon_0)}{m(m+1)} } \\
& = \Gamma^{(m-1)/m},
\end{align*}
or 
\[
\Gamma = h^{ 2m^2/(m^2-1) +  (1+ \epsilon_0) /(m+1)} \th^{2
  -m(m-1)/(m^2-1) -  (1+ \epsilon_0) /(m+1)  }.
\]
This means that for this value of $\Gamma$, our lower bound on $h
\Op_h(\hamvf (a))$
is 
\begin{align*}
& \Gamma 
x_0^{  -3
  -\epsilon_0 + (1+ \epsilon_0 )/m  } \\
& \quad = h^{ 2m^2/(m^2-1) +  (1+ \epsilon_0) /(m+1)} \th^{2
  -m(m-1)/(m^2-1) -  (1+ \epsilon_0) /(m+1)  }
x_0^{  -3
  -\epsilon_0 + (1+ \epsilon_0 )/m  } .
\end{align*}
Observe that the exponent of $h$ is $2 + \O(m^{-1})$ which can be made
smaller than $2 + \eta$ for any $\eta>0$ by taking $m$ large.

We also have to choose the parameter $\varpi(h)$.  For that we again
match lower bounds:

\begin{align*}
\frac{ h^{3/(m+1)} \th^{(m-1)/(m+1)}}{\varpi(h) } & = 
\Gamma^{1/m} (h/\th)^{  \alpha(1+\epsilon_0 ) ( m-1)/m  } x_0^{  -3
  -\epsilon_0 + (1+ \epsilon_0 )/m  } \th^{(2m-2)/m} \\
& =  h^{ 2m/(m^2-1) +  (1+ \epsilon_0) /(m+1)    } \\
& \quad \times \th^{4
-2/m  -(m-1)/(m^2-1) -  (1+ \epsilon_0) /(m+1)  } x_0^{  -3
  -\epsilon_0 + (1+ \epsilon_0 )/m  } ,
\end{align*}
or
\begin{align*}
\varpi(h) & = h^{3/(m+1)   -2m/(m^2-1) -  (1+ \epsilon_0) /(m+1)  } \\
& \quad \times \th^{ (m-1)/(m+1)   -4 +2/m
  +(m-1)/(m^2-1) +  (1+ \epsilon_0) /(m+1)  }    x_0^{ 3 +
  \epsilon_0 - (1+ \epsilon_0 )/m } .
\end{align*}
Taking $m$ sufficiently large yields $\varpi(h)$ satisfying
$h/\varpi(h) = o(1)$, so $\varpi(h)
\gg h$, as required to determine $x_0(h)$ and fix all the parameters.

All told, we have shown for a function $u(X)$ with semiclassical
wavefront set localized in a set $\{ | X | \leq \epsilon
(h/\th)^{-1/(m+1)}, | \Xi | \leq \epsilon (h/\th)^{-m/(m+1)} \}$ 
\begin{align*}
& h (h/\th)^{(m-1)/(m+1)} \lll \Op_{\th} (g) u, u \rrr \\
&\quad  \geq C \Gamma(h)  x_0^{ -3 -
  \epsilon_0 + (1 +
  \epsilon_0)/m  } 
\| u \|^2\\
& \quad \quad + h (h/\th)^{(m-1)/(m+1)} (\lll \Op_{\th} (g_2) u, u
\rrr + \lll \Op_{\th} (g_3) u, u \rrr )\\
& \quad \geq C 
h^{ 2m^2/(m^2-1) +  (1+ \epsilon_0) /(m+1)} \th^{2
  -m(m-1)/(m^2-1) -  (1+ \epsilon_0) /(m+1)  } \\
& \quad \quad \times 
x_0^{  -3
  -\epsilon_0 + (1+ \epsilon_0 )/m  } 
\| u \|^2 \\
& \quad \quad - C'_{N, \delta}
\th^{2m/(m+1)} h^{(N-2)m/(m+1) - \delta} h (h/\th)^{(m-1)/(m+1)} \| u \|^2 \\
& \quad \geq C''
h^{ 2m^2/(m^2-1) +  (1+ \epsilon_0) /(m+1)} \th^{2
  -m(m-1)/(m^2-1) -  (1+ \epsilon_0) /(m+1)  } \\
& \quad \quad \times
x_0^{  -3
  -\epsilon_0 + (1+ \epsilon_0 )/m  } 
\| u \|^2.
\end{align*}
We note that with these parameter values, of course $h/ \varpi = o(1)$, which
implies in turn that $x_0 = o(1)$, and since $h/\varpi \gg h^2$, the estimate \eqref{E:tV-controls} holds,
which closes the argument.

This concludes the study of the principal term in the commutator
expansion.  Of course we still have to control the lower order terms
in the commutator expansion, which we do in the following Lemma.

\begin{lemma}
\label{L:Q-comm-error-3a}
The symbol expansion of $[Q_1, a^w]$ in the $h$-Weyl calculus is of
the form
\begin{align*}
[Q_1, a^w] = & \Op_h^w \Bigg( \Big( 
\frac{i h}{2} \sigma ( D_x , D_\xi; D_y , D_\eta) \Big) (q_1(x, \xi)
a(y, \eta) - q_1(y, \eta) a ( x , \xi) ) |_{ x = y , \xi = \eta} \\
& + e (
x, \xi ) + r_3(x, \xi)\Bigg) ,
\end{align*}
where $e$ satisfies
\[
\Op_h^w(e)  
\ll h \Op_h(\hamvf (a)).
\]

\end{lemma}

\begin{proof}
Since everything is in the Weyl calculus, only the odd terms in the
exponential composition expansion are non-zero.  Hence the $h^2$ term
is zero in the Weyl expansion.  
Now according to Lemma \ref{l:err} and the standard $L^2$ continuity
theorem for $h$-pseudodifferential operators, we need to estimate a
finite number of 
derivatives of the error:
\begin{equation}
\label{E:error-10001}
 | \partial^{\gamma} e_2 | \leq C h^{3}
\sum_{ \gamma_1 + \gamma_2 = \gamma } 
 \sup_{ 
{{( x, \xi) \in T^* \RR }
\atop{ ( y , \eta) \in T^* \RR }}} \sup_{
|\rho | \leq M  \,, \rho \in \NN^{4} }
\left|
\Gamma_{\alpha, \beta, \rho, \gamma}(D)
( \sigma ( D) ) ^{3} q_1 ( x , \xi)  
a ( y, \eta ) 
\right| 
.
\end{equation}
However, since $q_1(x, \xi) = \xi^2 + V_{0,h}(x)$, we have
\[
D_x D_\xi q_1 = D_{\xi}^3 q_1 = 0,
\]
so that
\begin{align*}
\sigma(D)^3 & q_1(x, \xi) a(y, \eta) |_{ x = y , \xi = \eta} \\
& = D_x^3 q_1 D_\eta^3 a |_{ x = y , \xi = \eta} \\
& = -V_h'''(x) (\th/h)^{3m/(m+1)} \Lambda'''
((\th/h)^{m/(m+1)} \eta) \\
& \quad \times \Lambda((\th/h)^{1/(m+1)} y) \chi(y)
\chi(\eta)
+ r_3,
\end{align*}
where $r_3$ is supported in $\{ | (x, \xi) | \geq \delta_1 \}$.  
Owing to the cutoffs $\chi(y) \chi(\eta)$ in the definition of $a$
(and the corresponding implicit cutoffs in $q_1$), we only need to
estimate this error in compact sets.  The derivatives
$h^{\beta} \partial_\eta$ and $h^\alpha \partial_y$
preserve the order of $e_2$ in $h$ and increase the order in $\th$, while the other derivatives lead
to higher powers in $h/\th$ in the symbol expansion.  Hence we need only estimate $e_2$,
as the derivatives satisfy similar estimates.

In order to estimate $e_2$, we again use conjugation to
the $2$-parameter calculus, and at some point invoke the 0-Gevrey assumption.  We have
\[
\| \Op_h^w(e_2) u \| = \| T_{h, \th} \Op_h^w(e_2) T_{h, \th}^{-1}
T_{h, \th} u \| \leq \| T_{h, \th} \Op_h^w(e_2) T_{h, \th}^{-1}
\|_{L^2 \to L^2} \| u \|,
\]
by unitarity of $T_{h, \th}$.  But $T_{h, \th} \Op_h^w(e_2) T_{h,
  \th}^{-1} = \Op_{\th}^w(e_2 \circ \B )$ and
\begin{align*}
e_2 \circ \B & =  -h^3 V'''_h((h/\th)^{1/(m+1)}X) (\th/h)^{3m/(m+1)} \Lambda'''
(\Xi) \\
& \quad \times \Lambda(X) \chi(x)
\chi(\xi) + r_3 \circ \B  ,
\end{align*}
where $r_3$ is again microsupported away from the critical point
(coming from the derivatives on $\chi(x) \chi(\xi)$.  
We recall that $V_h'''(x) = V'''(x) + W_h'''(x)$, where $W_h(x) =
\Gamma(h) f(x/x_0)$.  As $f \in \Ci_c$, we know that 
\[
|W_h'''(x)| \leq C \Gamma x_0^{-3},
\]
and hence
\begin{align*}
| h^3 (\th/h)^{3m/(m+1)} & \Lambda'''
(\Xi) \Lambda(X) W_h''' ((h/\th)^{1/(m+1)}X)  (\chi(x)
\chi(\xi) | \\
& \leq C \Gamma h^{3/(m+1)} \th^{3m/(m+1)} x_0^{-3}.
\end{align*}

As for $V$, since $V' \in \GG^0_\tau$, for $x$ close to $0$
satisfying (in the rescaled coordinates)
\[
| X | \geq \left( \frac{h}{\th} \right)^{-1/(m+1) + \epsilon_1}, \,\,\,
\epsilon_1 >0, 
\]
we have
\begin{align}
| &h^{3/(m+1)} \th^{3m/(m+1)}  V'''((h/\th)^{1/(m+1)} X ) | \label{E:V-3-der}\\
& \leq C h^{3/(m+1)} \th^{3m/(m+1)} 
\left| \left( \frac{h}{\th} \right)^{1/(m+1)} X \right|^{-2 \tau} | V_0'
((h/\th)^{1/(m+1)} X ) | \notag \\
& \ll h^{2/(m+1) + \gamma} \th^{3m/(m+1)} |  V_0'
((h/\th)^{1/(m+1)} X ) | \notag
\end{align}
provided
\[
2 \tau \epsilon_1 \leq \frac{1}{m+1} - \gamma,
\]
for $\gamma>0$ (which of course implies we must have $\gamma <
1/(m+1)$).  
This can clearly be done for any $\tau < \infty$ by taking
$\epsilon_1>0$ sufficiently small.

We need to estimate \eqref{E:V-3-der} in terms of $V_0'(X) \lll \Xi
\rrr^{-1-\epsilon_0}$ as $(x, \xi)$ {\it and} $(y, \eta)$ vary in \eqref{E:error-10001}.
That means we need to worry about large $| \Xi|$.  If $| \Xi | \leq
\delta_1/2$, say, then \eqref{E:V-3-der} is trivially bounded by 
\[
h^{2/(m+1) + \gamma} \th^{3m/(m+1)} |  V_0'
((h/\th)^{1/(m+1)} X ) | \lll \Xi \rrr^{-1-\epsilon_0} .
\]
If 
\[
| \Xi | \geq \max \{ | X |^{1+ \epsilon_0}, \delta_1/2\},
\]
then the function $g_1 \geq c_{\delta_1}$, so there is nothing to
prove in this region.  On the other hand, if
\[
\frac{\delta_1}{2} \leq | \Xi | \leq | X|^{1 + \epsilon_0},
\]
then
\[
| \Xi |^{-1} \geq | X |^{-1-\epsilon_0},
\]
so that
\[
\lll \Xi \rrr^{-1-\epsilon_0} \geq \lll X \rrr^{-(1 + \epsilon_0)^2}
\geq \left( \frac{h}{\th} \right)^{\alpha (1 + \epsilon_0)^2}.
\]
Then \eqref{E:V-3-der} is bounded by
\begin{align*}
& C h^{2/(m+1) + \gamma} \th^{3m/(m+1)} |  V_0'
((h/\th)^{1/(m+1)} X ) \left( \frac{h}{\th} \right)^{-\alpha (1 +
  \epsilon_0)^2}| \lll \Xi \rrr^{-1-\epsilon_0} \\
& \ll C h^{1/(m+1)} \th^{m/(m+1)}  |  V_0'
((h/\th)^{1/(m+1)} X ) \lll \Xi \rrr^{-1-\epsilon_0} ,
\end{align*}
provided
\[
\frac{(1 + \epsilon_0)^2}{m+1} < \frac{1}{m+1} + \gamma,
\]
which is possible since we have already determined $\epsilon_0 \ll
1/(m+1)$ and the only restriction on $\gamma$ was $\gamma < 1/(m+1)$.

On the other hand, we
have for 
\[
| X | \leq \left( \frac{h}{\th} \right)^{-1/(m+1) + \epsilon_1}, \,\,\,
\epsilon_1 >0, 
\]
since $V'''(x) = \O(|x|^\infty)$, then
\[
| V''' ((h/\th)^{1/(m+1)} X ) | = \O(h^\infty).
\]

The error term must be estimated in terms of $h \hamvf (a)$.  Recall
that $| \Lambda'''(\Xi ) | \leq C \lll \Xi \rrr^{-1-\epsilon_0}$, so 
 we have
shown that the error is always controlled by
\[
o(h^{1/(m+1)} \th^{m/(m+1)} ) \lll \Xi \rrr^{-1-\epsilon_0} |\Lambda (
X) V_{0,h}' ((h/\th)^{1/(m+1)} X ) | + \O
(h^\infty) \ll h \hamvf (a).
\]

\end{proof}

Finally, we are able to put things together.  
Let $v=\varphi^w u,$ with $\varphi$ chosen to have support inside the
set where $\chi(x)\chi(\xi)=1;$ thus the terms $r$ and $r_3$ above are supported
away from the support of $\varphi.$  Then
Lemma \ref{L:Q-comm-error-3a} yields
\begin{align*}
& i\ang{[Q_1-z,a^w]v,v}\\
& \quad =h\ang{\Op_h^w(\hamvf(a))v,v}+\ang{\Op_h^w(e_2)u,u}
\\
&\quad \geq C \Gamma x_0^{ -3 -
  \epsilon_0 + (1 +
  \epsilon_0)/m  } \norm{v}^2 \\
& \quad = 
C   h^{ 2m^2/(m^2-1) +  3 \epsilon_0 /(m^2-m)} \th^{m/(m+1) - 3
  \epsilon_0 /(m^2-m)} x_0^{ -3 -
  \epsilon_0 + (1 +
  \epsilon_0)/m  }\norm{v}^2,
\end{align*}
 for
$\th$ sufficiently small.  Here we have used the previously computed
value of $\Gamma$.  
On the other hand, we certainly have
$$
\big\lvert \ang{[Q_1-z,a^w]v,v}\big\rvert \leq C \norm{(Q_1-z)v}\norm{v},
$$
hence 
\[
\norm{(Q_1-z)v} \geq C \Gamma x_0^{ -3 -
  \epsilon_0 + (1 +
  \epsilon_0)/m  } \norm{v}.
\]
We need yet compare $\tQ$ to $Q_1$:
\begin{align*}
& \norm{v} \\
& \leq C \Gamma^{-1} x_0^{ 3 +
  \epsilon_0 - (1 +
  \epsilon_0)/m  } \norm{(Q_1-z)v} \\
& \leq C \Gamma^{-1} x_0^{ 3 +
  \epsilon_0 - (1 +
  \epsilon_0)/m  } \left( \norm{(\tQ -z) v } +
    \norm{(V_{0,h} - V_0 ) v } \right) \\
& \leq C \Gamma^{-1} x_0^{ 3 +
  \epsilon_0 - (1 +
  \epsilon_0)/m  } \left( \norm{(\tQ -z) v } +
     \Gamma(h) \norm{v } \right) \\
& \leq C \Gamma^{-1} x_0^{ 3 +
  \epsilon_0 - (1 +
  \epsilon_0)/m  } \norm{(\tQ -z) v } +
    o(1) \norm{v } 
\end{align*}
provided that again $\epsilon_0$ is sufficiently small and $m$ is
sufficiently large.
Then the term with $\norm{v}$ can be moved to the
left hand side to get (now freezing $\th$ small and positive)
\begin{align*}
\norm{v} & \leq C \Gamma^{-1} x_0^{ 3 +
  \epsilon_0 - (1 +
  \epsilon_0)/m  } \norm{(\tQ -z) v } \\
& \leq  C 
h^{ -2m^2/(m^2-1) -  (1+ \epsilon_0) /(m+1)} \norm{(\tQ -z) v }
\\
& = C h^{-2-\eta} \norm{(\tQ -z) v }
\end{align*}
for $\eta = \O(m^{-1})$.  
This is \eqref{E:ml-inv-3a}.

Lastly, we show how to modify the preceding argument in the case of
Proposition \ref{P:ml-inv-3b}.  The main point is that the nonlinear
rescaling in $\Gamma$ (as part of $\lambda$) allows us to use that
$\Gamma^{1/(m+1)} \gg \Gamma$.  The first step is to modify the function $f$
and subsequently $W_h$ and $V_{0,h}$.  Since $V_0(x) \equiv 1$ on an
interval $x \in [-a,a]$, with $\pm V_0'(x) <0$ for $\pm x > a$, we
choose the point $x_0>0$ so that
\[
- x V_0'( x ) \geq \frac{ h}{\varpi(h)}, \,\, \, a+x_0  \leq  x  \leq
a + \epsilon,
\]
and similarly for $-a-\epsilon \leq x \leq -a-x_0$.  
Again we can assume that $| x_0 - a| = o(1)$.  Then choose $f \in \Ci_c(\reals) \cap \GG^0_\tau$ for some $\tau
< \infty$, with $f(x) = 1 - \frac{1}{2m}
x^{2m}$ for $| x | \leq a + x_0 $, and $f' (x) \leq 0$ or $x \geq 0$, $\supp
f \subset [-a-2x_0, a + 2x_0]$, satisfying
\[
| \partial_x^k f | \leq C_k |x_0|^{-k}.
\]
For
our next parameter, set $\tGamma(h) = c_0 \Gamma(h)$ for a small
constant $c_0>0$ to be determined, and $\Gamma(h)$ the parameter
computed in the case of the isolated infinitely degenerate maximum.  As before then we take 
\[
W_h(x) = \tGamma(h) f(x),
\]
and let
\[
V_{0,h}(x) = V_0(x) + W_h(x)
\]
We then follow the same arguments as in the proofs of Proposition
\ref{P:ml-inv-1} and \ref{P:ml-inv-3a}, noting that the ``smallness''
assumption on the support of the microlocal cutoff $\phi$ in the $x$
direction was to control lower order terms in Taylor expansions.  As
the function $V_0$ is constant and $f(x) = 1-x^{2m}$ on $[-a,a]$, the smallness assumption
translates into a small neighbourhood around $[-a,a]$.  Hence in Proposition
\ref{P:ml-inv-3b} we have assumed that $\supp \phi \subset
[-a-\epsilon, a + \epsilon]$.  All of the error terms are treated
similarly to the preceding proof.  The only changes to check are that,
since $f$ is no longer a function of $x/x_0$, and $\tGamma(h) = c_0
\Gamma$, we need to solve (in the previous notation):
\[
h (h/\th)^{(m-1)/(m+1)} \tGamma^{1/m} (h/\th)^{ \alpha(1+\epsilon_0 )
  ( m-1)/m } \th^{2(m-1)/m} \gg \tGamma,
\]
or
\[
h (h/\th)^{(m-1)/(m+1)}  (h/\th)^{ \alpha(1+\epsilon_0 ) ( m-1)/m   } \th^{2(m-1)/m} \gg (c_0\Gamma)^{(m-1)/m},
\]
which is true with our previous choice of $\Gamma$, provided $c_0>0$
is sufficiently small and independent of $h$.

As previously, we then have
\[
\norm{(Q_1-z)v} \geq C^{-1} c_0^{1/m} h^{ 2m^2/(m^2-1) +  3 \epsilon_0 /(m^2-m)} \th^{m/(m+1) - 3
  \epsilon_0 /(m^2-m)}  \norm{v}.
\]
In order to save some space, let us denote
\[
\omega(h) = h^{ 2m^2/(m^2-1) +  3 \epsilon_0 /(m^2-m)} \th^{m/(m+1) - 3
  \epsilon_0 /(m^2-m)}  .
\]
Comparing $\tQ$ to $Q_1$ now yields:
\begin{align*}
\norm{v} & \leq C c_0^{-1/m} \omega(h)^{-1}
    \norm{(Q_1-z)v} \\
& \leq C c_0^{-1/m} \omega(h)^{-1}\left( \norm{(\tQ -z) v } +
    \norm{(V_{0,h} - V_0 ) v } \right) \\
& \leq C c_0^{-1/m} \omega(h)^{-1}\left( \norm{(\tQ -z) v } + C c_0 \omega(h)
     \norm{v } \right) \\
& \leq C c_0^{-1/m} \omega(h)^{-1} \norm{(\tQ -z) v } +
   Cc_0^{(m-1)/m} \norm{v } .
\end{align*}
Freezing $\th>0$ and $c_0>0$ sufficiently small, 
the term with $\norm{v}$ can be moved to the
left hand side to get
\begin{align*}
\norm{v} \leq C h^{-2-\eta} \norm{(\tQ -z) v } \\
\end{align*}
with $\eta = \O(m^{-1})$ once again.  This is \eqref{E:ml-inv-3b}.

\end{proof}

\subsection{Infinitely degenerate and cylindrical inflection
  transmission trapping}

\label{SS:inf-deg-infl}

In this subsection, we study the microlocal spectral theory in a
neighbourhood of infinitely degenerate and cylindrical inflection
transmission trapping.  This is very similar to Subsection \ref{SS:unst-inf},
but now the potential is assumed to be monotonic in a neighbourhood of
the critical value.

We begin with the case where the potential has an isolated
infinitely degenerate critical point of inflection transmission type.
As in the previous subsection, we write $V(x) = A^{-2}(x) + h^2
V_1(x)$ and denote $V_0(x) = A^{-2}(x)$ to be the principal part of
the potential.  
Let us assume the point $x = 1$ is
an infinitely degenerate inflection point, so that locally near $x = 1$,
the potential takes the form
\[
V_0(x) \sim C_1^{-1} - (x-1)^{\infty}, 
\]
where $C_1>1$.  Of course the constant is arbitrary
(chosen to again agree with those in \cite{ChMe-lsm}).  Let us assume
that our potential satisfies $V_0'(x) \leq 0$ near $x = 1$, with
$V_0'(x)<0$ for $x \neq 1$ so that the critical point $x = 1$ is isolated.
The next Proposition says that in this case the microlocal resolvent is
bounded by $\O(h^{-2-\eta})$ for any $\eta>0$.  Let 
\[
\tQ = (hD_x)^2 + V(x) -z.
\]

\begin{proposition}
\label{P:ml-inv-4a}
For $\epsilon>0$ sufficiently small, let $\phi \in \s(T^* \reals)$
have compact support in $\{ |(x-1,\xi) |\leq \epsilon\}$.  Then for
any $\eta>0$, there
exists $C = C_{\epsilon,\eta}>0$ such that 
\begin{equation}
\label{E:ml-inv-4a}
\| \tQ \phi^w u \| \geq C_\epsilon {h^{2 + \eta}} \|
\phi^w u \|, \,\,\, z \in  [C_1^{-1}-\epsilon, C_1^{-1} + \epsilon].
\end{equation}
\end{proposition}

On the other hand, if $V_0'(x) \equiv 0$ on an interval, say $x-1 \in
[-a,a]$ with $V_0'(x) < 0$ for $x -1< -a$ and $x -1 > a$, we do not expect
anything better than Proposition \ref{P:ml-inv-4a}.  The
next Proposition says that this is exactly what we do get.  To fix an energy
level, assume $V_0 \equiv
C_1^{-1}$ on $[-a,a]$.  We again write
\[
\tQ = (hD_x)^2 + V(x) -z.
\]
\begin{proposition}
\label{P:ml-inv-4b}
For $\epsilon>0$ sufficiently small, let $\phi \in \s(T^* \reals)$
have compact support in $\{ |x-1| \leq a + \epsilon, \, |\xi| \leq \epsilon\}$.  Then for
any $\eta>0$, there
exists $C = C_{\epsilon,\eta}>0$ such that 
\begin{equation}
\label{E:ml-inv-4b}
\| \tQ \phi^w u \| \geq C_\epsilon {h^{2 + \eta}} \|
\phi^w u \|, \,\,\, z \in  [C_1^{-1}-\epsilon, C_1^{-1} + \epsilon].
\end{equation}
\end{proposition}

\begin{proof}

The proof of these Propositions is again very similar, so we put them together.
We will first prove Proposition \ref{P:ml-inv-4a}, and then point out how
the proof must be modified to get Proposition \ref{P:ml-inv-4b}.

The idea of the proof of Proposition \ref{P:ml-inv-4a} (and indeed Proposition
\ref{P:ml-inv-4b}) is to ``round off the corners'' in an $h$-dependent
fashion to obtain a {\it finitely degenerate} inflection point, and then mimic the proof of Proposition \ref{P:ml-inv-2}.

Choose a point $x_0 = x_0(h) >0$ and $\epsilon>0$ such that $x_0$ is
the smallest number so that 
\[
-  V_0'( x ) \geq \frac{ h}{\varpi(h)}, \,\, \, x_0 \leq | x -1| \leq \epsilon,
\]
where $\varpi(h)$ will be determined later.  
Similar considerations apply to choosing the parameters here as in the
previous subsection, but since we wrote that in excruciating detail,
we will leave out one or two details in this subsection.  Fix $m_2 \geq
1$, and choose also
an odd function $f \in \Ci_c([-2,2]) \cap \GG^0_\tau$ for some $\tau
< \infty$, with $f(x) =  - (x)^{2m_2 + 1}/(2m_2 +1)$ for $| x | \leq 1$ and
$f(x), f'(x) \leq 0$ for $0 \leq x \leq 2$.  For
another parameter $\Gamma(h)$ to be determined, let
\[
W_h(x) = \Gamma(h) f((x-1)/x_0),
\]
and let
\[
V_{0,h}(x) = V_0(x) + W_h(x)
\]
and
\[
V_h(x) = V(x) + W_h(x)
\]
(see Figure \ref{fig:Vh-2}).  The parameter $\Gamma(h)$ will be seen
to be a constant multiple of $h^{2+\eta }$, where $\eta >0$, $\eta =
\O( m_2^{-1} )$ as $m_2 \to \infty$ in the case of Proposition
\ref{P:ml-inv-4a}.  As in the previous subsection, $\Gamma(h)$ will be
a small constant times this power of $h$ in the case of Proposition  \ref{P:ml-inv-4b}.  
\begin{figure}
\hfill
\centerline{\input{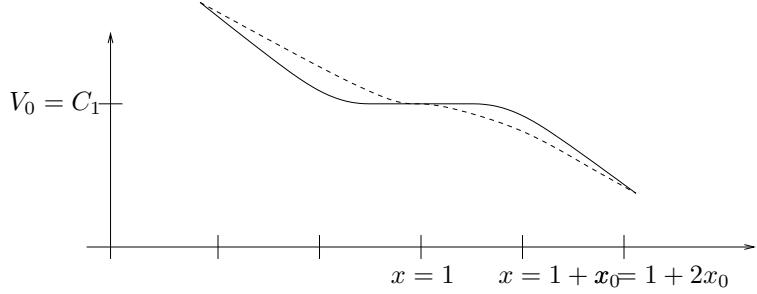}}
\caption{\label{fig:Vh-2} The potential $V_0$ and the modified potential
  $V_{0,h}$ (in dashed). }
\hfill
\end{figure}
By construction,
\[
| V_0(x) - V_{0,h}(x) | \leq | W_h | \leq \Gamma(h).
\]

Let $Q_1 = (hD)^2 + V_h$ with symbol $q_1 = \xi^2 + V_h$.  
The Hamilton vector field $\hamvf$ associated to the symbol $q_1$ is given by
\begin{align*}
\hamvf & = 2\xi\pa_{x} -V_h' \pa_{\xi} \\
& = 2 \xi \pa_x - \left( \frac{\Gamma(h)}{x_0} f'((x-1)/x_0) + V_0'(x)
  + h^2 V_1'(x) 
\right) \p_\xi .
\end{align*}

We will consider a commutant localizing in this region and singular at
the critical point 
in a controlled way: we introduce new variables
$$
\Xi=\frac{\xi}{(h/\th)^{\beta}},\quad X-1 = \frac{x-1}{(h/\th)^\alpha},
$$
with 
$\alpha, \beta > 0$, $\alpha = 1/(m_2 +1)$, and $\alpha + \beta = 1$ so that we may use the
two-parameter calculus.

We remark that in the new ``blown-up'' coordinates $\Xi,X,$
\begin{align}
\hamvf & = (h/\th)^{\beta - \alpha}\big(2\Xi \pa_X -
(h/\th)^{\alpha - 2 \beta } V'_h( (h/\th)^{\alpha} (X-1) +1)
\pa_\Xi\big) 
  \label{blownupvf-4}
\end{align}

Now fix $\epsilon_0>0$ and set 
$$
\Lambda_1(s) = \int_0^s \ang{s'}^{-1-\epsilon_0} \, ds'
$$
and
\[
\Lambda_2(s) =  1 + \int_{-\infty}^s \lll s' \rrr^{-1-\ep_0} \, ds'.
\]
$\Lambda_1$ is a bounded symbol which looks like $s$ near $0$, and
$\Lambda_2$ is a bounded symbol with positive derivative for $s$ near
$0$, and $\Lambda_2 \geq 1$ everywhere.

We introduce the singular symbol
\begin{align*}
a(x,\xi;h) & = \Lambda_1(\Xi)\Lambda_2(X-1)\chi(x-1)\chi(\xi)\\
& =
\Lambda_1(\xi/(h/\th)^{\beta}) \Lambda_2( (x-1)/(h/\th)^\alpha)\chi(x-1)\chi(\xi),
\end{align*}
where $\chi(s)$ is a cutoff function equal to $1$ for $\abs{s}\leq \delta_1$
and $0$ for $s\geq 2\delta_1$ ($\delta_1$ will be chosen shortly).  
Then $a$ is bounded since we have restricted the
domain of integration to $|(x-1, \xi) | \leq \delta_1$.  Further, $a$
satisfies the symbolic estimates:
$$
\abs{\pa_X^{\vec{\alpha}} \pa_\Xi^{\vec{\beta}} a}\leq C_{\vec{\alpha},\vec{\beta}} .
$$
(Recall that $x-1=(h/\th)^\alpha (X-1)$ and $\xi=(h/\th)^{\beta}\Xi.$)
Using \eqref{blownupvf-4}, it is simple to
compute
\begin{equation}\label{gdefn-4}
\begin{aligned}
\hamvf (a) = & (h/\th)^{\beta - \alpha}\chi(x-1)\chi(\xi)\big(
2\Lambda_1(\Xi)
\ang{X-1}^{-1 - \ep_0}\Xi \\
& - (h/\th)^{\alpha - 2 \beta } V'_h( (h/\th)^{\alpha} (X-1) +1)\ang{\Xi}^{-1-\ep_0}
\Lambda_2(X-1) \big)+r\\
=: & (h/\th)^{\beta - \alpha} g+r
\end{aligned}
\end{equation}
with $$\supp r\subset \{\abs{x-1}>\delta_1\} \cup \{\abs{\xi}>\delta_1\}$$
($r$ comes from terms involving derivatives of $\chi(x-1)\chi(\xi)$).

 For $|X-1| \leq (h/\th)^{-\alpha} x_0$ we have
\begin{align*}
- & (h/\th)^{\alpha - 2 \beta }  V'_h( (h/\th)^{\alpha} (X-1) +1)\ang{\Xi}^{-1-\ep_0}
\Lambda_2(X-1) \\
& = \Gamma(h) x_0^{-(2m_2+1)}  (h/\th)^{\alpha(2m_2+1) - 2 \beta}
(X-1)^{2m_2} \ang{\Xi}^{-1-\ep_0}
\Lambda_2(X-1)  + g_2 ,
\end{align*}
with
\[
g_2 = - (h/\th)^{\alpha - 2 \beta } ( V_0' + h^2 V_1')( (h/\th)^{\alpha} (X-1) +1)\ang{\Xi}^{-1-\ep_0}
\Lambda_2(X-1) .
\]
Let us denote by $g_1$ the part of $g$ obtained in this fashion,
microlocally in $\{ |X-1| \leq (h/\th)^{-\alpha} x_0 \}$:
\begin{align*}
g_1 & = g - g_2 \\
& = 2\Lambda_1(\Xi)
\ang{X-1}^{-1 - \ep_0}\Xi \\
& \quad + \Gamma(h) x_0^{-(2m_2+1)}  (h/\th)^{\alpha(2m_2+1) - 2 \beta}
(X-1)^{2m_2} \ang{\Xi}^{-1-\ep_0}
\Lambda_2(X-1) .
\end{align*}

 For $|X-1| \leq (h/\th)^{-\alpha} x_0$ and $|
\Xi | \leq (h / \th )^{-\beta} \delta_1$ consider 
\begin{align*}
g_1 
= & 2\Lambda_1(\Xi) \Xi \ang{X-1}^{-1-\ep_0} \\
& + 
\frac{\Gamma(h)}{x_0^{2m_2+1}}  (h/\th)^{\alpha(2m_2+1) - 2 \beta}
(X-1)^{2m_2}\Lambda_2(X-1)\ang{\Xi}^{-1-\ep_0} \\
= & 2\Lambda_1(\Xi) \Xi \ang{X-1}^{-1-\ep_0} \\
& + 
\frac{\Gamma(h)}{x_0^{2m_2+1}}  (h/\th)^{\alpha}
(X-1)^{2m_2}\Lambda_2(X-1)\ang{\Xi}^{-1-\ep_0} ,
\end{align*}
where we have used $\alpha = 1/(m_2 + 1)$.  Continuing, and rescaling using the $L^2$-unitary rescaling
\[
X' -1= \lambda (X-1), \,\,\, \Xi' = \lambda^{-1}\Xi,
\]
we get 
\begin{align*}
g_1
= & \lambda^2 \Big( 2\lambda^{-1}\Lambda_1 ( \Xi ) (\lambda^{-1}\Xi )
  \ang{X-1}^{-1-\ep_0} \\
& +  \lambda^{-2-2m_2} \Gamma x_0^{-2m_2-1} (h/\th)^{\alpha} \Lambda_2 (X-1) ( \lambda (X-1))^{2m_2} \ang{\Xi}^{-1-\ep_0} \Big) \\
= & \lambda^2 \Big( 2 \lambda^{-1}\Lambda_1( \lambda \Xi') \Xi' \ang{
    \lambda^{-1} (X'-1) }^{-1-\ep_0} \\
& + \lambda^{-2-2m_2} \Gamma x_0^{-2m_2-1} (h/\th)^{\alpha} \Lambda_2
  (\lambda^{-1}(X'-1)) (X'-1)^{2m_2} \ang{ \lambda \Xi' }^{-1-\ep_0} \Big) .
\end{align*}
As in the previous subsection, the parameter $\lambda>0$ will be seen
to be a small $h$-dependent parameter chosen to optimize lower bounds
on $g_1$ amongst several different regions.

The error term $g_2$ is the term in the expansion of $g$ coming
from $V'$ rather than $W_h'$.  We will deal with $g_2$ in due course.  
We are now microlocalized on a set where
\[
| X' -1| \leq \lambda (h/\th)^{-\alpha} x_0, \,\,\, | \Xi' | \leq \lambda^{-1} (h/\th)^{-\beta} \delta_1,
\]
and will be quantizing in the $\th$-Weyl calculus, so we need symbolic
estimates on these sets.  

If 
\[
| X'-1| \leq \lambda\delta_1, \text{ and } | \Xi' | \leq \lambda^{-1} \delta_1,
\]
and $\delta_1>0$ is sufficiently small, 
then $\Lambda_1( \lambda \Xi') \sim \lambda \Xi'$ and
$\Lambda_2(\lambda^{-1}( X'-1)  ) \sim 1$, so
that $g_1$ is bounded below by a multiple of
\begin{equation}
\min \{  \lambda^2,  \lambda^{-2-2m_2} \Gamma (h/\th)^{\alpha} \} (
(\Xi')^2 + (X'-1)^{2m_2} ). \label{E:g11-1}
\end{equation}
Hence the $\th$-quantization of $g_1$ is bounded below by this minimum
value 
times 
$\th^{2m_2/(m_2 +1)}$ on this set (using \cite[Lemma A.2]{ChWu-lsm}).

Now on the complementary set, 
if $\abs{\lambda \Xi'}\geq
 \max\left( \abs{ \lambda^{-1} (X'-1)}^{1+\ep_0}, (\delta_1/2)^{1 + \epsilon_0} \right)$
then
\begin{align*}
g_1& \geq c \lambda^2 \lambda^{-1} \Lambda_1( \lambda \Xi') \Xi' \ang{ \lambda^{-1} (X'-1) }^{-1-\ep_0} \\
& \geq c\lambda \sgn( \Xi') \Xi' | \lambda  \Xi' |^{-1} \\
& \geq c_0
\end{align*}
for some $c_0 >0$.

If $\abs{\lambda^{-1}  (X'-1) }^{1+\ep_0}\geq
\max\left(\abs{\lambda \Xi'}, (\delta_1 /2)^{1 + \ep_0}\right),$ we
have two regions to consider.  The first, if 
 $| \lambda \Xi'| \leq (\delta_1 /2)^{1 + \epsilon_0}$ (and using that $| \lambda^{-1}
X' | \leq (h/\th)^{-\alpha} x_0$ in this region),
then
\begin{align}
g_1 & \geq c \lambda^2 \Big(  ( \Xi' )^2 (h/\th)^{\alpha(1 + \epsilon_0)}
x_0^{-1-\epsilon_0}  \notag \\
& \quad + \lambda^{-2-2m_2} \Gamma x_0^{-2m_2 -1}
(h/\th)^{\alpha} 
   (X'-1)^{2m_2} 
\Big) \\
& \geq \min \{ \lambda^2 (h/\th)^{\alpha(1 + \epsilon_0)}
x_0^{-1-\epsilon_0} , \lambda^{-2m_2} \Gamma  x_0^{-2m_2 -1}
(h/\th)^{\alpha}  \} \notag \\
& \quad \times ( ( \Xi')^2 + (X'-1)^{2m_2} ). \label{E:g11-2}
\end{align}
We optimize this by setting the two terms in the minimum equal:
\[
\lambda^2 (h/\th)^{\alpha(1 + \epsilon_0)}
x_0^{-1-\epsilon_0} = \lambda^{-2m_2} \Gamma  x_0^{-2m_2 -1}
(h/\th)^{\alpha} ,
\]
or
\[
\lambda^{2 + 2m_2} = \Gamma (h/\th)^{-\alpha \epsilon_0} x_0^{-2m_2 +
  \epsilon_0},
\]
which yields in turn the lower bound
\begin{align*}
\lambda^2 & (h/\th)^{\alpha(1 + \epsilon_0)}
x_0^{-1-\epsilon_0} \\
& = \Gamma^{1/(m_2 +1)} (h/\th)^{-\alpha \epsilon_0 /
  (m_2 +1) + \alpha(1 + \epsilon_0)} x_0^{(-2m_2 + \epsilon_0 )/(m_2
  +1)  -1-\epsilon_0} .
\end{align*}
Then according to \cite[Lemma A.2]{ChWu-lsm}, the $\th$-quantization
of \eqref{E:g11-2} is bounded below by
\[
\Gamma^{1/(m_2 +1)} (h/\th)^{-\alpha \epsilon_0 /
  (m_2 +1) + \alpha(1 + \epsilon_0)} x_0^{(-2m_2 + \epsilon_0 )/(m_2
  +1)  -1-\epsilon_0} \th^{2m_2/(m_2 +1)}.
\]

On the other hand, if $| \lambda^{-1} (X'-1)|^{1 + \epsilon_0} \geq |
\lambda \Xi'| \geq (\delta_1/2)^{1 + \epsilon_0}$, then 
\begin{align}
g_1 & \geq  c  \sgn(\Xi') \lambda \Xi' 
(h/\th)^{\alpha(1 + \epsilon_0)} x_0^{-1-\epsilon_0} \notag  \\
& \geq c (h/\th)^{\alpha(1 + \epsilon_0)} x_0^{-1-\epsilon_0} . \label{E:g11-3}
\end{align}

We now again take the worst lower bound for
\eqref{E:g11-1}-\eqref{E:g11-3} to get for a function $u$ with
$h$-wavefront set localized in 
\[
|(X'-1)| \leq \lambda
(h/\th)^{-\alpha} x_0, \,\, | \Xi' | \leq \lambda^{-1}
(h/\th)^{-\beta} \delta_1,
\]
\begin{align*}
& \lll \Op_{\th} (g_1) u, u \rrr \\ & \quad \geq 
c \Gamma^{1/(m_2 +1)} 
h^{ (1 + \epsilon_0)/(m_2 +1)  - \epsilon_0 / (m_2 +1)^2} \\ & \quad
\quad \times 
\th^{ 2-  (3 + \epsilon_0)/(m_2 +1)  + \epsilon_0 / (m_2 +1)^2  }
x_0^{   -3  - \epsilon_0+ (2 + \epsilon_0)/(m_2 +1)   }
\| u \|^2.
\end{align*}
Here we have used that $\alpha = 1/(m_2 +1)$.

On the other hand, if $| \lambda^{-1}(X'-1) | \geq (h/\th)^{-\alpha} x_0$, we use the
assumed lower bound on $V_0'$ to estimate $g$ from below.  Examining
the potential terms, we have
\begin{align}
-& (h/\th)^{\alpha - 2 \beta} V'_h( (h/\th)^{\alpha} X)\ang{\Xi}^{-1-\ep_0}
\Lambda_2(X-1) \big) \notag \\ 
& =  - (h/\th)^{\alpha - 2 \beta} V'( (h/\th)^{\alpha} X)\ang{\Xi}^{-1-\ep_0}
\Lambda_2(X-1) \big) + g_3 \notag \\
& \geq C (h/\th)^{\alpha - 2 \beta}  \frac{h}{\varpi(h)}
(h/\th)^{(1 + \ep_0)\beta } + g_3 \notag \\
& = C \frac{ h^{2 \alpha + \ep_0 \beta} \th^{\beta - \alpha - \ep_0
    \beta}}{\varpi(h)} + g_3, \label{E:tV-controls-2}
\end{align}
assuming that $h/\varpi(h) \gg h^2$ so that $V_0'$ controls $h^2 V_1'$
(this will be verified later).
The error $g_3 \geq 0$ comes from using $V'$
in the expansion of $g$ rather than $W_h'$.

We now deal with the (nearly) positive error terms $g_2$ and $g_3$.  
\begin{lemma}
\label{L:g2g3-2}
The error terms $g_2$ and $g_3$ are semi-bounded below in the
following sense: if $u(X)$ has wavefront set localized in
\[
\{ |X-1| \leq \epsilon (h/\th)^{-\alpha}, \,\, | \Xi | \leq \epsilon
(h/\th)^{-\beta} \},
\]
then for any $\delta >0$ and $N >0$, 
\[
\lll \Op_{\th} ( g_j ) u, u \rrr \geq -C_N 
h^{(N+1) \alpha -
 2\beta -  \delta} \th^{2 \beta - \alpha} \| u \|^2,
\]
for $j = 2, 3$.

\end{lemma}

\begin{proof}

We prove the relevant bounds for $x \geq 1$.  The analysis for $x \leq
1$ is similar.  For $g_2$, for $N>0$ large, and $\delta>0$ small, choose $1 < x_1 <
x_2 = 1 + o(1)$ satisfying
\[
- V_0'(x_1) = h^{N \alpha}
\]
and
\[
- V_0'(x_2) = h^{N \alpha-\delta}.
\]
As usual, since $V_0'(x) = \O( (x-1)^\infty )$, the points $x_j$, $j = 1, 2$
satisfy $x_j -1 \gg h^{\delta_1}$ for any $\delta_1 >0$.  As before, the 0-Gevrey
condition also implies $|x_2-x_1| \gg h^{\delta_1}$ for any
$\delta_1>0$ as well.  

Now let $\psi(x)$ be a smooth function, $\psi \geq 0$, $\psi(x) \equiv 1$ on
$[1,x_1]$ with $\psi (x) = 0$ for $x \geq x_2$.  Assume also that
$| \partial_x^k \psi | \leq C_k| x_2 - x_1 |^{-k} = o(h^{-k \delta_1} )$
for any $\delta_1 >0$.  Let $\tpsi(X) = \psi((h/\th)^{\alpha} X )$ so
that 
\[
| \p_X^k \tpsi | \leq C_k (h/\th)^{\alpha k} | x_2 -
x_1|^{-k}=o(h^{k(\alpha - \delta_1)} \th^{-\alpha k}).
\] 
  We have
\begin{align*}
& \lll \Op_{\th} ( g_2 ) u, u \rrr \\
& \quad = \lll \Op_{\th} ( g_2 ) (1 - \tpsi)
u, (1 - \tpsi ) u \rrr + 
\lll \Op_{\th} ( g_2 ) \tpsi u, \tpsi u \rrr  \\
& \quad \quad + 2\lll \Op_{\th} ( g_2 )
\tpsi u, (1 -\tpsi ) u \rrr .
\end{align*}
We estimate each term separately.

On the support of $1-\tpsi$ (recalling once again that we are
restricting our attention to $x \geq 1$), we have $(h/\th)^{\alpha} X \geq x_1$ so
that in this region we can once again appeal to the 0-Gevrey condition
to control $V_1'$.  As discussed previously, $V_1$ consists of
quotients of derivatives of the function $A$ with powers of $A$.  As
$A$ is bounded above and below for $x$ small, we again use the 0-Gevrey
condition to write for any $\delta_1>0$ and for some $s, \tau <
\infty$, 
\begin{align*}
h^2 | V_1 '((h/\th)^\alpha (X-1) + 1) | & \leq C h^2 |x_1 |^{-s \tau }
| A'((h/\th)^\alpha (X-1) + 1) | \\
& \leq C h^{2-s \tau \delta_1} | V_0'((h/\th)^\alpha (X-1) + 1) |,
\end{align*}
and similarly for a finite number of derivatives of $V_1$.  Taking
$\delta_1>0$ sufficiently small, we see that on the support of $1 -
\tpsi$, $V_0'$ controls $h^2 V_1'$.  That is, for $h>0$ sufficiently
small,
\begin{align*}
& \lll \Op_{\th} (g_2) (1 - \tpsi) u, (1-\tpsi) u \rrr \\
& \quad = - (h/\th)^{\alpha - 2 \beta } \lll \Op_{\th} (( V_0' + h^2 V_1')\ang{\Xi}^{-1-\ep_0}
\Lambda_2(X-1) ) (1 - \tpsi) u, (1-\tpsi) u \rrr \\
& \quad \geq - \frac{1}{2} (h/\th)^{\alpha - 2 \beta } \lll \Op_{\th} ( V_0'\ang{\Xi}^{-1-\ep_0}
\Lambda_2(X-1) ) (1 - \tpsi) u, (1-\tpsi) u \rrr .
\end{align*}
Here to save space (and since it will be integrated out anyway) we
suppressed the argument of $V_0$ and $V_1$; both functions are
understood to be evaluated at $( (h/\th)^{\alpha} (X-1) +1)$.  
Then in this same region we have:
\begin{align*}
 -(h/\th)^{\alpha - 2 \beta}  
    & \Lambda_2(X-1)  
  V_0'(   (h/\th)^{\alpha} (X-1) +1  )  \lll \Xi
\rrr^{-1-\ep_0} \\
& =  (h/\th)^{\alpha - 2 \beta}  h^{N \alpha } A(X, h,
\th) \lll \Xi \rrr^{-1-\ep_0}
\end{align*}
where $A$ is a symbol bounded below by a positive constant.  On the
set where $A \lll \Xi \rrr^{-1-\ep_0} \geq 1$, this operator is
bounded below, while on the complement, we use the Sharp
G\r{a}rding inequality to get for any $\delta_1 >0$
\[
\lll \Op_{\th} ( g_2 ) (1 - \tpsi)
u, (1 - \tpsi ) u \rrr \geq -C_{\delta_1} \th h^{N \alpha + 2\alpha - 2 \beta
  -\delta_1} \th^{2 \beta - 2 \alpha} \| (1 - \tpsi)u \|^2.
\]
For the remaining two terms, on the support of $\tpsi$, we have
$1 \leq (h/\th)^{\alpha} X \leq x_2$, so that in order to estimate
$V_1'$, we need to estimate a finite number of derivatives of $A$ from
above.  But we know that for $k \geq 1$, $| \p_x^k A |$ is an
increasing function for $x \geq 1$ in a small neighbourhood.  Hence we
can estimate the size of $V_1'$ by estimating it at $x_2$.  For this,
we once again use the 0-Gevrey assumption to get for any $\delta_1>0$
and for some $s, \tau  < \infty$
\begin{align*}
h^2 | V_1'( ( h/\th)^\alpha (X-1) + 1) | & \leq C h^2 | x_2|^{-s \tau }
| V_0 '(x_2) | \\
& \leq C h^{2 - \delta_1 s \tau} h^{Nm/(m+1) - \delta},
\end{align*}
by our choice of $x_2$.  This shows that on the support of $\tpsi$,
$h^2 V_1'$ is controlled by a large power of $h$.
Then in this region
\begin{align*}
g_2 & = -(h/\th)^{\alpha - 2 \beta} 
    \Lambda_2(X) 
   V'((h/\th)^{\alpha} X )  \lll \Xi
\rrr^{-1-\ep_0} \\
& =  (h/\th)^{\alpha - 2 \beta}  h^{N\alpha - \delta} A_1(X, h,
\th) \lll \Xi \rrr^{-1-\ep_0},
\end{align*}
where $A_1$ is a function satisfying
\[
| \partial_X^k A_1 | \leq C_{k, \delta_1} (h^{\alpha - \delta_1 } \th^{-\alpha}
)^k.
\]
Hence if $\delta_1 < \alpha$, 
\[
\lll \Op_{\th} ( g_2 ) \tpsi u, \tpsi u \rrr  = \O(h^{(N+1) \alpha -
 2\beta -  \delta} \th^{2 \beta - \alpha} ) \| u \|^2,
\]
and similarly
\[
\lll \Op_{\th} ( g_2 )
\tpsi u, \tpsi u \rrr = \O(h^{(N+1) \alpha -
 2\beta -  \delta} \th^{2 \beta - \alpha} ) \| u \|^2.
\]

As in Lemma \ref{L:g2g3}, the proof for $g_3$ is the same.

\end{proof}

We are now in position again to fix some of the parameters.  We start
with $\Gamma$, which we again want to be much smaller than our
computed lower bound on $h \Op_h(\hamvf (a) )$.  We need to solve
\begin{align*}
h&  (h/\th)^{(m_2-1)/(m_2+1)} \Gamma^{1/(m_2 +1)} 
h^{ (1 + \epsilon_0)/(m_2 +1)  - \epsilon_0 / (m_2 +1)^2} \\
& \quad \times \th^{ 2-  (3 + \epsilon_0)/(m_2 +1)  + \epsilon_0 / (m_2 +1)^2  }
x_0^{   -3  - \epsilon_0+ (2 + \epsilon_0)/(m_2 +1)   } \\ & \gg \Gamma
\end{align*}
Again, $m_2>0$ will be large, $\epsilon_0>0$ is small, and $x_0>0$ is
$o(1)$, so it suffices to solve
\begin{align*}
& h^{2 
 - (1 - \epsilon_0)/(m_2 +1)  - \epsilon_0 / (m_2 +1)^2} \\
& \quad \times 
\th^{ 2    -(m_2-1)/(m_2+1)   -  (3 + \epsilon_0)/(m_2 +1)  + \epsilon_0 / (m_2 +1)^2  }
\\
& \quad \quad = \Gamma^{m_2/(m_2+1)},
\end{align*}
or
\[
\Gamma = h^{2 + 1/m_2 
  + \epsilon_0/m_2   - \epsilon_0 / m_2 (m_2 +1)} 
\th^{ 1   -   \epsilon_0/m_2   +
  \epsilon_0 / m_2 (m_2 +1)  }
\]
Then for this value of $\Gamma$, our lower bound on $h
\Op_h(\hamvf(a))$ is
\begin{align*}
\Gamma & x_0^{   -3  - \epsilon_0+ (2 + \epsilon_0)/(m_2 +1)   }  \\
& = h^{2 + 2/m_2 
 - (1 - \epsilon_0)/m_2   - \epsilon_0 / m_2 (m_2 +1)} \\
& \quad \times
\th^{ 2 + 2/m_2   -(m_2-1)/m_2   -  (3 + \epsilon_0)/m_2   +
  \epsilon_0 / m_2 (m_2 +1)  } \\
& \quad \times x_0^{   -3  - \epsilon_0+ (2 +
  \epsilon_0)/(m_2 +1)   } .
\end{align*}
We again observe that in this case, the exponent of $h$ is $2 + \O(
m_2^{-1})$, which can be made smaller than $2 + \eta$ for any
$\eta>0$.

We can now choose the parameter $\varpi(h)$ as well, again by
matching:
\begin{align*}
& \frac{ h^{2 \alpha + \ep_0 \beta} \th^{\beta - \alpha - \ep_0
    \beta}}{\varpi(h)} \\
& \quad =
\Gamma^{1/(m_2 +1)} 
h^{ (1 + \epsilon_0)/(m_2 +1)  - \epsilon_0 / (m_2 +1)^2} \\
& \quad \quad \times 
\th^{ 2-  (3 + \epsilon_0)/(m_2 +1)  + \epsilon_0 / (m_2 +1)^2  }
x_0^{   -3  - \epsilon_0+ (2 + \epsilon_0)/(m_2 +1)   } \\
& \quad = h^{  \frac{3m_2+1}{m_2^2 + m_2} + \frac{\epsilon_0}{m_2 +1}   } \\
& \quad \quad \times 
\th^{ 2 - \frac{2}{m_2+1} - \frac{\epsilon_0}{m_2+1}   } \\
& \quad \quad \times x_0^{   -3  - \epsilon_0+ \frac{(2 + \epsilon_0)}{(m_2
  +1)}   } ,
\end{align*}
or
\begin{align*}
\varpi(h) & = h^{\gamma_1    } 
\th^{\gamma_2 } x_0^{   3  + \epsilon_0- (2 + \epsilon_0)/(m_2
  +1)   },
\end{align*}
where
\begin{align*}
\gamma_1 & = -\frac{1}{m_2} + \epsilon_0 \left( \frac{m_2-1}{m_2+1} \right)
\end{align*}
and
\begin{align*}
\gamma_2 & = -1 - \epsilon_0 \left( \frac{m_2-1}{m_2+1} \right) .
\end{align*}

All told, we have shown for a function $u(X)$ with semiclassical
wavefront set localized in a set $\{ | X -1| \leq \epsilon
(h/\th)^{-\alpha }, | \Xi | \leq \epsilon (h/\th)^{-\beta} \}$ (again
using Lemma \ref{L:g2g3-2} to bound the $g_2$ and $g_3$ terms)
\begin{align*}
& h (h/\th)^{(m_2 -1) / (m_2 +1) }\lll \Op_{\th} (g) u, u \rrr \\
& \quad  \geq c \Gamma(h) x_0^{  -3  - \epsilon_0+ (2 + \epsilon_0)/(m_2 +1)  } \| u \|^2\\
& \quad \quad + h (h/\th)^{(m_2 -1) / (m_2 +1) } (\lll \Op_{\th} (g_2) u, u
\rrr + \lll \Op_{\th} (g_3) u, u \rrr )\\
& \quad \geq c (1 - o(1)) h^{2 + 2/m_2 
 - (1 - \epsilon_0)/m_2   - \epsilon_0 / m_2 (m_2 +1)} \\
& \quad \quad \times 
\th^{ 2 + 2/m_2   -(m_2-1)/m_2   -  (3 + \epsilon_0)/m_2   +
  \epsilon_0 / m_2 (m_2 +1)  } \\
& \quad \quad \times x_0^{   -3  - \epsilon_0+ (2 +
  \epsilon_0)/(m_2 +1)   } \| u \|^2 .
\end{align*}

As in the previous subsection, we have $h / \varpi = o(1)$, so that
$x_0 = o(1)$, and $h / \varpi \gg h^2$ so that \eqref{E:tV-controls-2} holds, which
closes this part of the argument.

This concludes the study of the principal term in the commutator
expansion.  Of course we still have to control the lower order terms
in the commutator expansion, which we do in the following Lemma.  This
is where it becomes very important that $\alpha = 1/(m_2 +1) >0$.

\begin{lemma}
\label{L:Q-comm-error-4a}
The symbol expansion of $[Q_1, a^w]$ in the $h$-Weyl calculus is of
the form
\begin{align*}
[Q_1, a^w] = & \Op_h^w \Bigg( \Big( 
\frac{i h}{2} \sigma ( D_x , D_\xi; D_y , D_\eta) \Big) (q_1(x, \xi)
a(y, \eta) - q_1(y, \eta) a ( x , \xi) ) |_{ x = y , \xi = \eta} \\
& + e (
x, \xi ) + r_3(x, \xi)\Bigg) ,
\end{align*}
where $e$ satisfies
\[
\Op_h^w(e)  
\ll h \Op_h(\hamvf (a)).
\]

\end{lemma}

\begin{proof}
Since everything is in the Weyl calculus, only the odd terms in the
exponential composition expansion are non-zero.  Hence the $h^2$ term
is zero in the Weyl expansion.  
Now according to Lemma \ref{l:err} and the standard $L^2$ continuity
theorem for $h$-pseudodifferential operators, we need to estimate a
finite number of 
derivatives of the error:
\[
 | \partial^{\gamma} e_2 | \leq C h^{3}
\sum_{ \gamma_1 + \gamma_2 = \gamma } 
 \sup_{ 
{{( x, \xi) \in T^* \RR }
\atop{ ( y , \eta) \in T^* \RR }}} \sup_{
|\rho | \leq M  \,, \rho \in \NN^{4} }
\left|
\Gamma_{\alpha, \beta, \rho, \gamma}(D)
( \sigma ( D) ) ^{3} q_1 ( x , \xi)  
a ( y, \eta ) 
\right|  
.
\]
However, since $q_1(x, \xi) = \xi^2 + V_h(x)$, we have
\[
D_x D_\xi q_1 = D_{\xi}^3 q_1 = 0,
\]
so that
\begin{align*}
\sigma(D)^3 & q_1(x, \xi) a(y, \eta) |_{ x = y , \xi = \eta} \\
& = D_x^3 q_1 D_\eta^3 a |_{ x = y , \xi = \eta} \\
& = -V_h'''(x) (\th/h)^{2 \beta} \Lambda_1'''
((\th/h)^{\beta} \eta) \\
& \quad \times \Lambda_2((\th/h)^{\alpha} y) \chi(y)
\chi(\eta)
+ r_3,
\end{align*}
where $r_3$ is supported in $\{ | (x, \xi) | \geq \delta_1 \}$.  
Owing to the cutoffs $\chi(y) \chi(\eta)$ in the definition of $a$
(and the corresponding implicit cutoffs in $q_1$), we only need to
estimate this error in compact sets.  The derivatives
$h^{\beta} \partial_\eta$ and $h^\alpha \partial_y$
preserve the order of $e_2$ in $h$ and increase the order in $\th$, while the other derivatives lead
to higher powers in $h/\th$ in the symbol expansion.  Hence we need only estimate $e_2$,
as the derivatives satisfy similar estimates.

In order to estimate $e_2$, we again use conjugation to
the $2$-parameter calculus, and at some point invoke the 0-Gevrey assumption.  We have
\[
\| \Op_h^w(e_2) u \| = \| T_{h, \th} \Op_h^w(e_2) T_{h, \th}^{-1}
T_{h, \th} u \| \leq \| T_{h, \th} \Op_h^w(e_2) T_{h, \th}^{-1}
\|_{L^2 \to L^2} \| u \|,
\]
by unitarity of $T_{h, \th}$.  But $T_{h, \th} \Op_h^w(e_2) T_{h,
  \th}^{-1} = \Op_{\th}^w(e_2 \circ \B )$ and
\begin{align*}
e_2 \circ \B & =  -h^3 V'''_h((h/\th)^{\alpha}X) (\th/h)^{3\beta } \Lambda_1'''
(\Xi) \\
& \quad \times \Lambda_2(X) \chi(x)
\chi(\xi) + r_3 \circ \B  ,
\end{align*}
where $r_3$ is again microsupported away from the critical point
(coming from the derivatives on $\chi(x) \chi(\xi)$.  
We recall that $V_h'''(x) = V'''(x) + W_h'''(x)$, where $W_h(x) =
\Gamma(h) f((x-1)/x_0)$.  As $f \in \Ci_c$, and we have already
computed the value $\Gamma$, we know that 
\[
|W_h'''(x)| \leq C \Gamma x_0^{-3},
\]
and hence
\begin{align*}
| h^3 (\th/h)^{3\beta} \Lambda_1'''
(\Xi) \Lambda_2(X) W_h''' ((h/\th)^{\alpha}X)  (\chi(x)
\chi(\xi) | & \leq C\Gamma h^{ 3 \alpha }\th^{3\beta} x_0^{-3} \\
& = o(\Gamma(h) x_0^{  -3  - \epsilon_0+ (2 + \epsilon_0)/(m_2 +1)  }),
\end{align*}
which is little-o of our computed lower bound on the quantization  $h \Op_h
(\hamvf (a))$.  

As for $V$, since $V' \in \GG^0_\tau$, for $x$ close to $1$
satisfying (in the rescaled coordinates)
\[
| X-1 | \geq \left( \frac{h}{\th} \right)^{-\alpha + \epsilon_1}, \,\,\,
\epsilon_1 >0, 
\]
\begin{align*}
| h^{3\alpha} \th^{3\beta} V'''((h/\th)^{\alpha} X ) | & \leq Ch^{3\alpha} \th^{3\beta} 
\left| \left( \frac{h}{\th} \right)^{\alpha} (X-1) \right|^{-2 \tau} | V'
((h/\th)^{\alpha} X ) | \\
& \ll h^{2\alpha + \gamma} \th^{3\beta} |  V'
((h/\th)^{\alpha} X ) | 
\end{align*}
provided $\epsilon_1>0$ is sufficiently small (as in the proof of
Lemma \ref{L:Q-comm-error-3a}).  We also follow again the proof of Lemma
\ref{L:Q-comm-error-3a} to break the analysis into several regions.
The details are precisely the same.

On the other hand, we
have for 
\[
| X-1 | \leq \left( \frac{h}{\th} \right)^{-\alpha+ \epsilon_1}, \,\,\,
\epsilon_1 >0, 
\]
since $V_0'''(x) = \O(|x-1|^\infty)$, then
\[
| V_0''' ((h/\th)^{\alpha} X ) | = \O(h^\infty).
\]

The error term must be estimated in terms of $h \hamvf (a)$; we have
shown that the error is always controlled by
\begin{align*}
& o(\Gamma(h) x_0^{  -3  - \epsilon_0+ (2 + \epsilon_0)/(m_2 +1)  })
\\
& \quad + 
o(h^{\alpha} \th^\beta) | V_{0,h}' ((h/\th)^{\alpha} X ) \lll \Xi \rrr^{-1-\ep_0}| + \O
(h^\infty) ,
\end{align*}
which, when quantized, is controlled by little-o of our computed lower
bound on the quantization $h \Op_h( \hamvf (a) )$.

\end{proof}

The rest of the proof of Proposition \ref{P:ml-inv-4a} follows exactly as in
the proof of Proposition \ref{P:ml-inv-3a}.

Lastly, we show how to modify the preceding argument in the case of
Proposition \ref{P:ml-inv-4b}.  The first step is to modify the function $f$
and subsequently $W_h$ and $V_{0,h}$.  Since $V_0(x) \equiv 1$ on an
interval $x-1 \in [-a,a]$, with $V_0'(x) <0$ for $\pm (x-1) > a$, we
choose the point $x_0>0$ to be the smallest number so that
\[
-  V_0'( x ) \geq \frac{ h}{\varpi(h)}, \,\, \, x_0 \leq x-1 -a \leq \epsilon
\]
and similarly for $-\epsilon \leq x - 1 + a \leq -x_0$.  We choose
also the same parameter $\varpi(h)$ as for Proposition
\ref{P:ml-inv-4b}.  
Again we can assume that $x_0  = o(1)$.  Then choose $f \in \Ci_c(\reals) \cap \GG^0_\tau$ for some $\tau
< \infty$, with $f(x) =  - 
x^{2m_2 +1}$ for $| x | \leq a + x_0$, and $f' (x) \leq 0$ for $x \geq 0$, $\supp
f \subset [-a-2x_0, a + 2x_0]$, satisfying
\[
| \partial_x^k f | \leq C_k x_0^{-k}.
\]
Following the proof of Proposition \ref{P:ml-inv-3b}, we set 
the parameter $\tGamma(h) = c_0 \Gamma(h)$ for a small
constant $c_0>0$ to be determined, and where $\Gamma(h)$ was computed
in the course of the proof of Proposition \ref{P:ml-inv-4a}.  As in the
proof of Proposition \ref{P:ml-inv-3b}, we then take 
\[
W_h(x) = \tGamma(h) f((x-1)),
\]
and let
\[
V_h(x) = V(x) + W_h(x)
\]
We then follow the same arguments as in the proofs of Propositions
\ref{P:ml-inv-3a} and \ref{P:ml-inv-4a}, noting that the ``smallness''
assumption on the support of the microlocal cutoff $\phi$ in the $x$
direction was to control lower order terms in Taylor expansions.  As
the function $V_0$ is constant and $f(x) = 1-x^{2m_2+1}$ on $[-a,a]$, the smallness assumption
translates into a small neighbourhood around $[1-a,1+a]$.  Hence in
Proposition 
\ref{P:ml-inv-4b} we have assumed that $\supp \phi \subset
[1-a-\epsilon, 1+a + \epsilon]$.  All of the error terms are treated
similarly to the preceding proof.  Then the same rescaling argument as
in the proof of Proposition \ref{P:ml-inv-3b} proves Proposition \ref{P:ml-inv-4b}.

\end{proof}

\subsection{Stable trapping and quasimodes}

\label{SS:stable}

Suppose $V_0(x)$ has an ``honest'' local minimum at $x = 0$ in the sense
that $V_0$ is eventually increasing as one moves to the left or right of $x = 0$.
That is, 
$V_0(0) = \Vmin$, $V_0'(0) = 0$, $\pm V_0'(x) \geq 0$ for $\pm x \geq
0$ near $0$, and if 
\[
\begin{cases}
x_+ \geq 0 \\
x_- \leq 0
\end{cases}
\]
are the smallest positive/negative values with $V_0'(x) \neq 0$ for
$\pm x > \pm x_\pm$, then $\pm V'(x) >0$ for $\pm x > \pm x_\pm$ in
some neighbourhood.  This means 
there exists $\delta>0$ such that for each $y \in
[\Vmin, \Vmin+ 2 \delta]$, the sets $\{ x : V_0(x) \leq y \}$ have a
non-empty compact connected component containing $x=0$.  By shrinking
$\delta>0$ if necessary, we may also assume that $V_0(x)$ is a convex
function on the connected component containing $x = 0$.

Let $-a < 0$ be the
largest negative number such that $V_0(-a) = \Vmin + \delta$, and let $b>0$ be
the smallest positive number such that $V_0(b) = \Vmin + \delta$.  By
again shrinking $\delta>0$ if necessary, we may assume that $V_0'(-a) <0$ and
$V_0'(b) >0$.  
Choose also $\epsilon>0$ such that $V_0(-a-\epsilon) \geq \Vmin + 3\delta/2$
and $V_0(b+\epsilon) \geq \Vmin + 3\delta/2$ (again shrinking
$\delta>0$ if necessary).
Figure \ref{fig:V-stab} is a picture of the setup.

\begin{figure}
\hfill
\centerline{\input{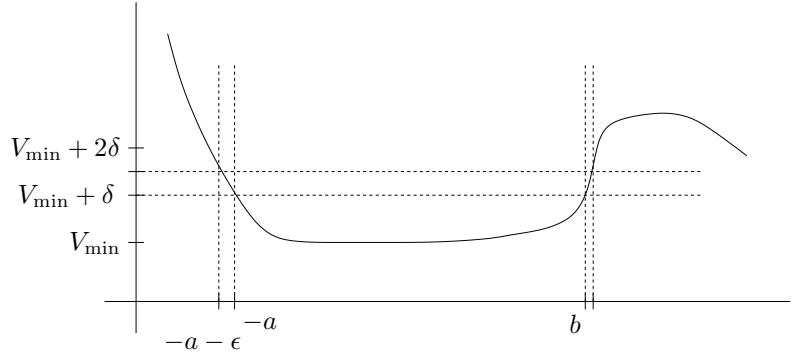}}
\caption{\label{fig:V-stab} The potential $V_0$ near a local min, and
  the choice of parameters $\delta$, $-a$, $b$, and $\epsilon$. }
\hfill
\end{figure}

Let
\[
\tV(x) = \begin{cases} V(x) = V_0(x) + h^2 V_1(x), x \in [-a-\epsilon, b + \epsilon], \\
  \beta x^2, | x | \gg 1,
\end{cases}
\]
where $\beta>0$ is an appropriate constant so that $\tV$ can be
assumed convex.  In particular, we may assume that $\tV^{-1}( E )
\subset [-a, b]$ for $E \in [\Vmin +\delta/2, \Vmin
+ 2 \delta
/3]$, by taking $h>0$ sufficiently small.

Let $L = (hD_x)^2 + \tV(x)$.  
For $h>0$
sufficiently small, Weyl's law implies there exists $\sim h^{-1}$
eigenvalues $E$ of the operator $L$ in the interval $E \in [\Vmin
+\delta/2, \Vmin + 
2 \delta/3]$.  Fix such an eigenvalue, and let $\phi(x)$ be the
(normalized) 
associated eigenfunction.  Then, since $E \in [\Vmin
+\delta/2, \Vmin
+
2 \delta/3]$, we have in particular that 
\[
\phi(x) = \O(h^\infty), \,\,\, x \leq -a, \, x \geq b.
\]
Let $\chi(x) \in \Ci_c( \reals)$ be a smooth function such that
$\chi(x) \equiv 1$ on $[-a,b]$ with support in $[-\epsilon - a, b +
\epsilon]$.  Thus,
\[
\chi(x) \phi(x) = \phi(x) + \O(h^\infty),
\]
and
\begin{align*}
L \chi \phi & = \chi L \phi + [L, \chi] \phi \\
& = E \chi \phi + \O(h^\infty),
\end{align*}
since $[L, \chi]$ is supported on the set where $\phi =
\O(h^\infty)$.  

But since $\tV = V$ on the set $[-a-\epsilon, b + \epsilon]$, we also
have
\[
((hD_x)^2 + V(x) ) \chi \phi = E \chi \phi + \O(h^\infty).
\]
As $\| \chi \phi \| = 1 - \O(h^\infty)$, we have 
\[
\| ((hD_x)^2 + V(x) -E ) \chi \phi \|  = \O(h^\infty) \| \chi \phi \|,
\]
and so evidently for any $N$, there exists $C_N$ such that for any
compactly supported function $\tchi$ such that $\tchi = 1$ on $\supp
\chi$,
\begin{equation}
\label{E:stable-qm}
\| \tchi ((hD_x)^2 + V(x) -E)^{-1} \tchi \chi \phi \| \geq C_N h^{-N}
\| \chi \phi \|.
\end{equation}

\begin{remark}
Of course, if we know more about the structure of the function
$V_0(x)$ near an honest local minimum, then we can say more.  For
example, we can construct WKB approximations as quasimodes, and then
stationary phase can tell us much more detailed information about the quasimodes.  However
indirect, the Weyl law method presented here is in a sense more
robust, and does not really require intimate knowledge of $V_0(x)$
near the minimum.  

\end{remark}

\section{Proof of Theorem \ref{T:dichotomy}}

We are now able to prove Theorem \ref{T:dichotomy}.  Since we have
assumed that the trapped set has only finitely many connected
components, this implies that the function $A(x)$ has only finitely
many critical values, which consequently occur in a compact set.  Let
$A_1, A_2, \ldots A_k$ be the critical values, and let $K_l = \{ A(x)
= A_l \}$ be the critical sets.  
There are two cases to consider.

{\bf Case 1:}  The function $V_0(x) = A^{-2}(x)$ has a local minimum.  Then
apply  \eqref{E:stable-qm} to conclude there are highly localized
quasimodes, and the resolvent therefore blows up faster than any
polynomial (at least along a subsequence).

{\bf Case 2:}  The function $V_0(x) = A^{-2}(x)$ has no minima.  In this case,
each critical value of the function $A^{-2}(x)$ is of either unstable
or transmission inflection type (whether infinitely degenerate or
not).  The important thing to observe is that there are only finitely
many critical values, and they are all isolated in the two-dimensional
phase space in the following sense: if $K_l$ is disconnected for some $l$,
then assume $K_l$ has only two connected components (the finite case
being similar).  The connected components of $K_l$ are separated by a maximum, say
$K_j$ (or
minimum, but this would be Case 1).  The stable and unstable manifolds
associated to the flow around $K_j$ form a {\it global} separatrix,
separating the complete flowouts of the two components of $K_l$.  This
allows us to microlocalize and glue together the trapping estimates examined
in detail in Subsections \ref{SS:unst-nd}-\ref{SS:inf-deg-infl} (see
Appendix \ref{A:gluing}, 
\cite{Chr-disp-1}, \cite{ChMe-lsm}, and \cite{DaVa-gluing}).  The rest of the proof
proceeds precisely as in the proof of \cite[Theorem 2]{ChWu-lsm}.

\section{Proof of Corollary \ref{C:l-sm-cor}}
\label{S:l-sm-pf-sec}

In this section, we recall the functional theoretic argument which
connects resolvent estimates to local smoothing estimates, and in the
process prove Corollary \ref{C:l-sm-cor}.  The technique is often
called a ``$T T^*$'' argument, however we have already used $T$ in our
time interval, so instead we use an $A A^*$ argument.

Let $u_0 \in \s(X)$ and $\chi \in \Ci_c(X)$ and let $A$ be the operator
\be
A u_0 = \chi e^{it\Delta} u_0,
\ee
acting on $L^2(X)$.  We want to show
\be
A : L^2(X) \to L^2([0,T]; H^{s}(X))
\ee 
for some $s>0$ 
is bounded.  By duality, this is equivalent to the adjoint $A^*$ being
bounded 
\be
A^* : L^2([0,T]; H^{-s}(X)) \to L^2(X),
\ee
which is equivalent to the boundedness of the composition
\be
A A^* : L^2([0,T]; H^{-s}(X)) \to L^2([0,T]; H^{s}(X)).
\ee
Computing directly, we get 
\be
A A^* f(t) = \int_0^T \chi e^{i  (t - \tau) \Delta} \chi f( \tau) d
\tau.
\ee

Now let $u$ be defined by 
\be
u(x,t) = \int_0^T e^{i  (t - \tau) \Delta } \chi f( \tau) d
\tau.
\ee
Since we are only interested in the
time interval $[0,T]$, we extend $f$ to be $0$ for $t \notin [0,T]$.
We write
\be
A A^* f(t) & = & \int_0^t \chi e^{i  (t - \tau) \Delta} \chi f( \tau) d
\tau + \int_t^T \chi e^{i  (t - \tau) \Delta} \chi f( \tau) d
\tau \\
& = :& \chi u_1(t) + \chi u_2(t),
\ee
and calculate
\ben
\label{uj-eqn}
(D_t -\Delta ) u_j = (-1)^j i \chi f.
\een
Thus boundedness of $A A^*$ will follow if we prove $u$ satisfying
\eqref{uj-eqn} satisfies
\be
\| \chi u \|_{L^2([0,T];H^{s})} \leq \| f
\|_{L^2([0,T];H^{-s})}.
\ee
Replacing $\pm i f$ with $f$ in equation \eqref{uj-eqn} and taking the
Fourier transform in time, $t \mapsto z$, results in the following
equation for $\hat u$ and $\hat f$:
\ben
\label{ft-u-eqn}
(z -\Delta ) \hat u(z, \cdot ) = \chi \hat f (z, \cdot).
\een
Since $f(t, \cdot)$ is supported only in $[0,T]$,
$\hat f(z, \cdot)$ and $\hat u(z, \cdot)$ 
are holomorphic, bounded, and satisfy \eqref{ft-u-eqn} in $\{ \Im z <0 \}$.  Let $z =
\tau - i \eta$, $\eta >0$ sufficiently small.  Since the Fourier
transform is an $L^2 H$ isometry for any Hilbert space $H$, we want to estimate 
\be
\| \chi \hat u(z, \cdot) \|_{H^{s}(X)} \leq C \|
 \hat f(z, \cdot) \|_{H^{-s}(X)}
\ee
uniformly in $z$.

For this, we observe that if we know
\[
\| \chi ( -\Delta +z)^{-1} \chi \|_{L^2 \to L^2} \leq C | z |^{-r}
\]
for some $r \geq 0$ and $\Im z = -\eta<0$ fixed, then a standard interpolation argument gives
\[
\| \chi ( -\Delta +z)^{-1} \chi \|_{L^2 \to H^2} \leq C | z |^{1-r},
\]
and hence 
\[
\| \chi ( -\Delta +z)^{-1} \chi \|_{L^2 \to H^{2r}} \leq C .
\]
Interpolating again, we get
\[
\| \chi ( -\Delta +z)^{-1} \chi \|_{H^{-r} \to H^r} \leq C .
\]

 Thus
\be
\| \chi  u \|_{L^2( [0,T] ; H^{r}(X))} & \leq &
e^{\eta T} \| e^{-\eta t} \chi u(t) \|_{L^2( [0,T] ;
  H^{r}(X))} \\
& \leq & Ce^{\eta T} \| \chi \hat{u}(\tau - i \eta) \|_{L^2(
  \reals ; H^{r}(X))} \\
& \leq & Ce^{\eta T} \| \hat f( \tau - i \eta) \|_{L^2( \reals ; H^{-r}(X))} \\
& \leq & C e^{\eta T} \|e^{- \eta t} f(t) \|_{L^2([0,T];  H^{-r}(X))} \\
& \leq & C e^{\eta T} \| f(t) \|_{L^2([0,T];  H^{-r}(X))}.
\ee
Hence 
\be
\int_0^T \|\chi u \|_{H^{r}(X)}^2 dt \leq C e^{ \eta
  T} \int_0^T \| f \|_{H^{-r}(X)}^2 dt,
\ee
or $A A^*$ is bounded.

We remark in passing that this argument works for any $r \geq 0$,
along a strip where $\Im z = - \eta <0$.  If we examine the case where
$\|\chi ( - \Delta +z )^{-1} \chi \|_{L^2 \to L^2}$ blows up as $\Im z \to
0$ (as in Case 2 of Theorem \ref{T:dichotomy}), we can use instead the trivial bound
\[
\|\chi (z - \Delta ) \chi \|_{L^2 \to L^2} \leq \frac{1}{| \Im z | } =
\frac{1}{\eta}
\]
in this case to get a zero derivative smoothing effect.  But of course
we already knew such an estimate must be true (even without spatial
cutoffs) from the $L^2(X)$ conservation law.  The point is that the
blowup of the resolvent is perfectly consistent with our physical
intuition in this problem.

\section{An Application: Spreading of Quasimodes for some Partially
  Rectangular Billiards}

In this section, we apply the microlocal estimates proved in the
previous sections to prove a spreading result for rather weak quasimodes for the
Laplacian in partially rectangular billiards.  The main result is that
if a partially rectangular billiard opens ``outward'' in at least one
wing, then any $\O( \lambda^{-\epsilon})$ quasimode must spread to
outside of any $\O(\lambda^{-\epsilon})$ neighbourhood of the
rectangular part.  This result holds for any $\epsilon>0$.

These results are similar in
spirit to results of Burq-Zworski \cite{BuZw-bb} and of Burq-Hassell-Wunsch
\cite{BHW-spread}, but the techniques of proof
are different.  
Let us be precise.

Let $\Omega \subset \reals^2$ be a planar domain with boundary in the
0-Gevrey class $\GG^0_\tau$ for $\tau < \infty$, and let $R = [-a,a]
\times [-\pi, \pi] \subset \reals^2$ be a rectangle with boundary
consisting of the two sets of parallel segments $\p R = \Gamma_1 \cup
\Gamma_2$, with $\Gamma_1 = [-a,a] \times \{ \pi \} \cup [-a,a] \times
\{ - \pi \}$.  Assume $R \subset \Omega$ and $\Gamma_1 \subset \p
\Omega$ but $\mathring{\Gamma}_2 \cap \p \Omega = \emptyset$.  We assume that for $(x, y)$ in a neighbourhood of $R$, $\p
\Omega$ is symmetric about the line $y = 0$.  Let $Y(x) = \pi + r(x)$ be a graph parametrization of the boundary curve $\p \Omega$ for
$(x,y)$ near $[-a,a] \times \{ \pi \}$.

\begin{theorem}
\label{T:billiards}
Consider the quasimode problem on $\Omega$:
\[
\begin{cases}
(-\Delta - \lambda^2) u = E(\lambda) \| u \|_{L^2}, \text{ on }
\Omega, \\
B u = 0, \text{ on } \p \Omega,
\end{cases}
\]
where $B = I$ or $B = \p_\nu$ (either Dirichlet or Neumann boundary
conditions).

Assume that $\pm r'(x)>0$ for at least one of $\pm (x \mp a) >0$ (that is, the
boundary curves ``outward'' away from the rectangular part of the
boundary for at least one side).  Fix $\epsilon>0$.  If $E(\lambda) = \O(\lambda^{-\epsilon})$ as $\lambda \to \infty$ and $\WF_{\lambda^{-1}}
u$ vanishes outside a neighbourhood of size
$\O(\lambda^{-\epsilon})$ of $R$, then $u = \O(\lambda^{-\infty})$ on
$\Omega$.

\end{theorem}

See Figures \ref{fig:billiard1} and \ref{fig:billiard2} for examples where the theorem applies.

\begin{remark}

It should be clear from the proof that the 0-Gevrey assumption need only
hold in a neighbourhood of the rectangular region $R$.  It should also
be clear from the proof that the symmetry in $y$ is not necessary; it
is enough that the vertical distance is increasing on at least one
side of the rectangular part (see Figure \ref{fig:billiard3}).

It is believed that there  can be a sequence of eigenfunctions which
concentrate on the entire rectangular part.  This result does not
preclude this, as the neighbourhood in which the theorem applies
shrinks to the rectangular part as $\lambda \to \infty$.  However, it
gives a lower bound on how fast a quasimode may concentrate on the
rectangular part under the assumptions of the Theorem.  
Moreover, the proof is
meant to be extremely elementary given the estimates established in
the first part of this paper.  
It is 
possible that in some special cases, with a little more care, the assumptions can be weakened
to include any quasimode localized in a sufficiently small
neighbourhood independent of $\lambda$.

We also remark that this theorem does not apply to the famous
Bunimovich stadium, since the boundary in that case is neither 0-Gevrey
smooth, nor does it open outward on either side of the rectangular
part.  Indeed, following the first part of the proof of Theorem
\ref{T:billiards} below, the effective potential curves ``upward'' as
$y \sim - (\pi - (x\mp a )^2 )^{1/2}$, where $2\pi$ is the height of the
Bunimovich stadium and $2a$ is the width of the rectangular part.  A very sketchy heuristic is that the lowest
energy quasimode sitting in this potential well occurs when the
potential well is approximately $\lambda^{-1}$ deep.  That is, it
should be concentrated in the set where
\[
y + \pi = \pi - (\pi - (x\mp a )^2 )^{1/2} \sim \lambda^{-1},
\]
which occurs when $x \mp a \sim \lambda^{-1/2}$, or within a
$\lambda^{-1/2}$ neighbourhood of the rectangular part.  This is
precisely the type of behaviour that Theorem \ref{T:billiards} rules out.

\end{remark}

\begin{figure}
\hfill
\centerline{\input{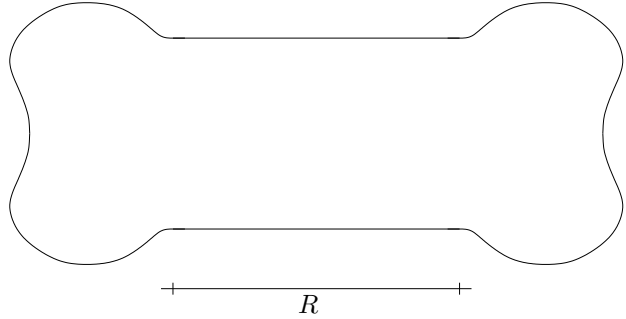}}
\caption{\label{fig:billiard1} A partially rectangular billiard
  opening ``outward'' away from the rectangular part.  No $\O(\lambda^{-\epsilon})$
  quasimode can concentrate in an $\O(\lambda^{-\epsilon})$
  neighbourhood of $R$.}
\hfill
\end{figure}

\begin{figure}
\hfill
\centerline{\input{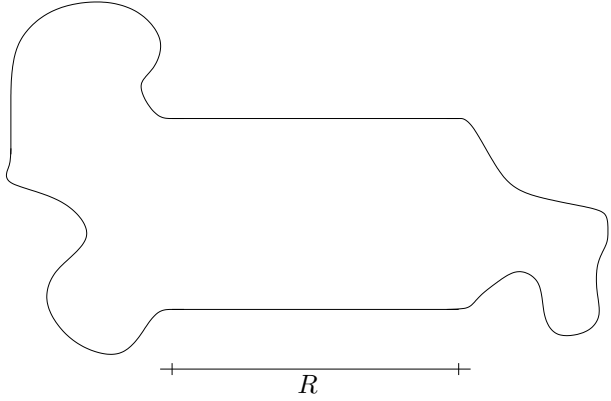}}
\caption{\label{fig:billiard2} A partially rectangular billiard
  opening ``outward'' away from the rectangular part on one side and
  ``inward'' on the other.  The same conclusion applies.}
\hfill
\end{figure}

\begin{figure}
\hfill
\centerline{\input{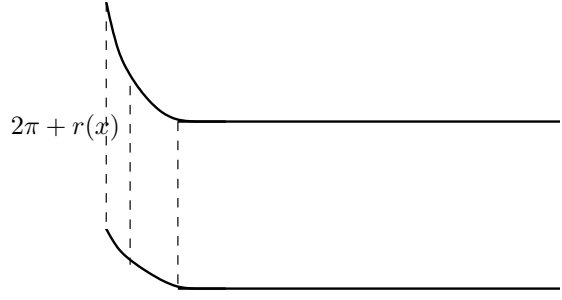}}
\caption{\label{fig:billiard3} Part of a partially rectangular
  billiard where the local symmetry in $y$ does not hold, but the
  vertical distance on one side is an increasing function away from
  the rectangular part.  The same conclusion applies in this case as
  well by trivially modifying the proof.}
\hfill
\end{figure}

\begin{proof}

The first step in the proof is to straighten the boundary near
the rectangular part and then approximately separate variables.  The
boundary $\Gamma$ near $R$ is given by $y = \pm Y(x) = \pm ( \pi +
r(x) )$ for $x \in [-a-\delta, a + \delta]$ for some $\delta>0$.
Write $P_0 = -\p_x^2 -\p_y^2$ for the flat Laplacian.  We will
``straighten the boundary'' near $R$ and compute the corresponding
change in the metric.  From this we will get a non-flat
Laplace-Beltrami operator which is almost separable.  Recall $u$
solves
\[
P_0 u = \lambda^2 u + E(\lambda) \| u \|
\]
and if $\chi(x, \lambda)$ satisfies $\chi \equiv 1$ on $\{ | x | \leq
a + \lambda^{-\epsilon} \}$, with support in, say $\{ | x | \leq a + 2
\lambda^{-\epsilon} \}$ then
\[
\chi u = u + \O(\lambda^{-\infty}) \| u \| 
\]
in any Sobolev space.  Let $\tR = \{ (x, y) \in \Omega : | x | \leq a
+ 4 \lambda^{-\epsilon} \}$ be a shrinking neighbourhood of $R$ in
$\Omega$, which is slightly larger than the set where $\chi u$ is supported.  We
change variables $(x, y) \mapsto (x', y')$  in $\tR$ in a way which straightens out the
boundary:
\[
\begin{cases}
x = x', \\
y =  y' Y(x').
\end{cases}
\]
Thus when $y = \pm Y(x) = \pm Y(x')$, $y' = \pm 1$.  We have
\begin{align*}
g & = dx^2 + dy^2 \\
& = (dx')^2 + (Y dy' + y' Y'(x') dx')^2 \\
& = (1 + A)(dx')^2 + 2B dx' dy' + Y^2 (dy')^2,
\end{align*}
where 
\[
A = (y' Y'(x'))^2,
\]
and
\[
B = y' Y' Y.
\]
In matrix notation, 
\[
g = \left( \begin{array}{cc} 1 + A & B \\ B &  Y^2 \end{array}
\right).
\]
Let us drop the cumbersome $(x', y')$ notation and write $(x,y)$
instead.  In order to compute $\Delta_g$ in these coordinates, we need
$| g|$ and $g^{-1}$.  We compute
\begin{align*}
| g | & = Y^2 (1 + A) - B^2 \\
& = Y^2 + y^2 Y^2 (Y')^2 - y^2 Y^2 (Y')^2 \\
& = Y^2.
\end{align*}
Hence
\[
g^{-1} = Y^{-2} \left( \begin{array}{cc} Y^2 & - B \\ -B &  1 + A \end{array}
\right).
\]
For our quasimode $u$ as above, we have after a tedious computation 
\begin{align*}
-\Delta_g u & = -\Big( \p_x^2 + Y^{-2} (1 + A) \p_y^2 + Y' Y^{-1} \p_x
- 2B Y^{-2} \p_x \p_y \\
& \quad 
-Y^{-1} (B/Y)_x \p_y -Y^{-1} (B/Y)_y \p_x + Y^{-1} ((1 + A)/Y)_y \p_y
\Big) u .
\end{align*}
Now let $\tchi (x, \lambda)$ be a smooth function such that $\tchi
\equiv 1$ on $\supp \chi$ with support in $\{ | x | \leq a + 4
\lambda^{-\epsilon} \}$.  Let us also assume for simplicity that we
have normalized $\| u \| = 1$.  Then, since $\Delta_g$ does not enlarge the
wavefront set, we have
\begin{align*}
-\Delta_g u & = -\tchi \Delta_g u + \O( \lambda^{-\infty} )  \\
& = - \Delta_g \tchi u - [\tchi, \Delta_g] u + \O( \lambda^{-\infty} )
 \\
& = - \Delta_g \tchi u + \O( \lambda^{-\infty} ) 
\end{align*}
by the support properties of $\tchi$, $\chi$, and the wavefront
assumption on $u$.  Hence we have
\[
-\Delta_g \tchi u = \lambda^2 \tchi u + E(\lambda) \| \tchi u \| + \O(\lambda^{-\infty}) .
\]
Observe now that the functions $A$ and $B$ are both $\O( Y'(x)) =
\O(r'(x))$, which for $(x, y) \in \tR$ is $\O( \lambda^{-\infty})$.
Hence,
\[
-\Delta_g \tchi u = P_2 \tchi u + \O( \lambda^{-\infty}) ,
\]
where $P_2 = -\p_x^2 -Y^{-2}(x) \p_y^2$.  That is,
\[
P_2 \tchi u = \lambda^2 \tchi u + E( \lambda ) \| \tchi u \| +
\O(\lambda^{-\infty} ) .
\]
Since $\tchi u$ is supported in $\tR$, which in these coordinates is
just the rectangle
\[
\tR = [-a-4\lambda^{-\epsilon}, a + 4 \lambda^{-\epsilon} ]_x \times
[-1,1]_y,
\]
we can expand in a Fourier basis (with appropriate boundary conditions):
\[
\tchi u = \sum_{k \in \ZZ} \tchi u_k(x) e_k(y).
\]
Let $\beta_k^2 \sim k^2$ be the eigenvalues in the $y$ direction, so that
$-\p_y^2 e_k = \beta_k^2 e_k$, and let $P_k = -\p_x^2 + \beta_k^2
Y^{-2}(x)$, so that
\[
P_2 \tchi u = \sum_{k \in \ZZ} e_k(y) P_k \tchi u_k.
\]
Now we rescale $h = \beta_k^{-1}$ and write
\[
P(h) = -h^2 \p_x^2 + Y^{-2}(x),
\]
so that $w = \tchi u_k$ must satisfy a semiclassical equation of the form
\[
P(h) w = z w + \tE \| w \| +  \O(\lambda^{-\infty} ) ,
\]
where $z = h^2 \lambda^2$ and $\tE = h^2 E(\lambda)$.

Now on $\tR$, the function $Y^{-2}$ satisfies $\pi^{-2} - \delta/2
\leq Y^{-2} \leq \pi^{-2}$ for some $\delta>0$ independent of $\lambda
\to \infty$.  If $z < \pi^{-2} - \delta$, say, then $P(h) - z$ is
elliptic and satisfies
\[
\| (P(h)-z)^{-1} \tchi \|_{L^2 \to L^2} \leq C.
\]
Hence for $z$ in this range, we have
\[
\| w \| =  C \tE \| w \| + \O( \lambda^{-\infty} ) ,
\]
which implies $\| w \| = \O( \lambda^{-\infty} ),$ since in this range
of $z$, we have $h^2 \leq C \lambda^{-2}$, which implies
\[
\tE = h^2 E(\lambda) = \O( \lambda^{-2} )
\]
in any case.  

Now if $z \geq \pi^{-2} + \delta$, then $\{ \xi^2 + Y^{-2}(x) = z \}$
has no critical points, so for these values of $z$, $P(h) -z$ obeys a
non-trapping estimate:
\[
\| (P(h) -z)^{-1} \tchi \|_{L^2 \to L^2} \leq C h^{-1}.
\]
Applying the same argument as in the previous case and observing that 
\[
h^{-1} \tE = h E = \O( \lambda^{-\epsilon} )
\]
yields again $w = \O( \lambda^{-\infty} )$.

For the remaining range of $z$, we must use our estimates from Propositions
\ref{P:ml-inv-3b} and \ref{P:ml-inv-4b}.  
We have assumed that $E(\lambda) = \O (
\lambda^{-\epsilon} ) = \O( h^\epsilon)$ for some $\epsilon >0$
fixed.

Since we are now in the region where 
\[
\pi^{-2} - \delta \leq z \leq \pi^{-2} + \delta,
\]
we have $h \sim \lambda^{-1}$, so $E(\lambda) = \O (
\lambda^{-\epsilon}) \sim \O( h^\epsilon)$ as $h \to 0$,
 and $\O(h^\infty)$ is equivalent to $\O( \lambda^{-\infty})$.  We
can further microlocalize and apply Proposition \ref{P:ml-inv-3b} or
\ref{P:ml-inv-4b} (as the case may be) with $\eta \ll \epsilon$ to get 
\[
\| w \| \leq C \lambda^{\eta} E(\lambda) \| w \| + \O( \lambda^{-\infty} ) ,
\]
which again implies $\| w \| = \O(\lambda^{-\infty} )$.

\end{proof}

\appendix

\section{A User's Guide to Resolvent Gluing}

\label{A:gluing}
Recently, a number of authors have used various constructions to glue
together resolvent estimates from different situations (see, for
example, \cite{Chr-disp-1,ChWu-lsm,DaVa-gluing,ChMe-lsm} and the
present work).  For example,
the best estimates near classically trapped sets are typically very
local (or microlocal) in nature, but in most reasonable situations,
the geodesic flow tends to infinity uniformly outside a small
neighbourhood of the trapping.  Hence one expects the behaviour at
spatial infinity to act somewhat independently of the behaviour near
the trapped set.  As a result, one tries to ``glue'' the microlocal
estimates near the trapping into the non-trapping estimates at
infinity.  In this appendix, we present a simple gluing technique
which works in the cases of interest in this paper; namely in one
dimensional semiclassical potential scattering.

Let $P = -h^2 \p^2 + V(x)$ be a semiclassical Schr\"odinger operator
in one spatial dimension.
We assume the potential $V(x)$ is a short range perturbation of the
inverse square potential (with one or two ``ends'').  That is, we
assume that for some $R >0$,
\[
| x | \geq R \implies | \p^k (V(x) - x^{-2} ) | \leq C_k \lll x
\rrr^{-2 -k}.
\]
Let $p = \xi^2 + V(x)$ be the symbol of $P$.  In this case, it is well
known that any critical points of $H_p$ are contained in a compact
set, and there exists a large, positive number $M$ and a symbol $\tp$
which is globally non-trapping, and $p = \tp$ for $x \geq M-1$.  Then
there are a wide selection of non-trapping estimates available for
$\tP = \Op_h (\tp)$ with a $\O(h^{-1})$ bound on the cutoff resolvent.  As we
have seen in this paper, if there are any stable critical elements of
$H_p$, then there are many nearby trajectories which do not escape to
infinity, and no gluing techniques are necessary, since we already
know the resolvent blows up rapidly.  Hence this construction will
only apply when all the critical elements of $H_p$ are at least weakly
unstable.  

Increasing $M$ later on if necessary, we first select a number of
cutoffs.  Let $\tchi \in \Ci_c( \reals)$, $0 \leq \tchi \leq 1$,
$\tchi \equiv 1$ for $| x | \leq 2M$.  Let $\tchi_2 \in \Ci_c(
\reals)$ be equal to $1$ on $\supp \tchi$ with slightly larger support.
Let $\rho_s \in \Ci(\reals)$ be
a smooth function, $\rho_s >0$, $\rho_s(x) \equiv 1$ on $\supp
\tchi_2$, $\rho(x) = \lll x \rrr^s$ for very large $|x|$.  Let $\Gamma
\in \Ci_c( \reals)$ be a cutoff equal to $1$ on $\{ | x | \leq M-1 \}$
with support in $\{ | x | \leq M \}$.  That means that the
non-trapping symbol $\tp$ equals $p$ on the support of $1-\Gamma$.
The idea is that things are well-behaved on the support of $1-\Gamma$,
and on the support of $\Gamma$, we decompose $\Gamma$ into a sum of
cutoffs where we have ``black-box'' microlocal estimates established
through other means.  The most important tool for gluing all of these
estimates together is the propagation of singularities lemma in the
next subsection.

\subsection{Propagation of Singularities}

The rough idea behind propagation of singularities is that if two
regions in phase space are connected by the $H_p$ flow, then the $L^2$
mass in one region is controlled by the $L^2$ mass in the other,
modulo a term involving $P$.  The very nature of requiring the $H_p$
flow to move from one region to another means these estimates  are inherently
non-trapping.

In order to motivate the more general statement below, let us describe
a baby version of propagation of singularities in the very special
case that the operator is $P = hD_x$.  Of course, it is well known
that if $H_p \neq 0$, then $P$ is microlocally conjugate to $hD_x$, so
this is not as ridiculous as it initially seems.  Indeed, the rough
heuristic we sketch here can be fixed up to provide an alternative
proof to Lemma \ref{wf-lemma-2} below.  The proof of this Lemma in
\cite{Chr-disp-1} proceeds by the traditional method of the original
proof of H\"ormander (see the original in \cite{Hor-sing} or the
presentation in \cite{Tay-pdo}).

Consider a function $u(x)$ of one variable.  Let $a < b$ be two points
in $\reals$.  We show that the $L^2$ mass of $u$ in a neighbourhood of
size $1$ about $a$ is controlled by the mass of $u$ near $b$ in a
neighbourhood of size $K$ modulo a term involving $P$.  Perhaps more
importantly, the constant on the term near $b$ is comparable to
$K^{-1/2}$.  That is, by enlarging the control region to size $K$, we
can make the constants in our estimates small.

For $s,t \geq 0$, we write
\[
u(a+s) - u(b+t) = \int_{a+s}^{b+t} u'(r) dr = \frac{i}{h}
\int_{a+s}^{b+t} P u dr.
\]
Rearranging and taking the absolute value squared and applying
H\"older's inequality to the integral, we get
\begin{align*}
| u(a+s)|^2 & \leq 2 | u(b+t)|^2 + 2 h^{-2} \left( \int_{a+s}^{b+t} | P u |
  dr \right)^2 \\
& \leq  2 | u(b+t)|^2 + 2 h^{-2} ((b+t) - (a+s) )  \| P u \|_{L^2(a+s, b+t)}^2 .
\end{align*}
We now integrate in $0 \leq s \leq 1$:
\[
\| u \|_{L^2(a,a+1)}^2 \leq 2 | u(b+t) |^2 + 2 h^{-2} ((b+t) - a )  \|
P u \|_{L^2(a, b+t)}^2.
\]
We follow this by integrating in $0 \leq t \leq K$ to get
\[
K \| u \|_{L^2(a,a+1)}^2 \leq 2 \| u \|_{L^2(b, b+K)}^2 + 2 K h^{-2} (b +K- a )  \|
P u \|_{L^2(a, b+K)}^2.
\]
We conclude that
\[
 \| u \|_{L^2(a,a+1)} \leq \frac{\sqrt{2}}{\sqrt{K}} \| u \|_{L^2(b,
   b+K)} + \sqrt{2} h^{-1} (b +K - a )^{1/2}  \|
P u \|_{L^2(a, b+K)}.
\]

The more general version of this idea is given in the following Lemma
from \cite{Chr-disp-1}.

\begin{lemma}
\label{wf-lemma-2}
Let $\widetilde{V}_1, \widetilde{V}_2
\Subset M$, and for $j = 1,2$ let $V_j \Subset T^*M$,
\be
V_j := \{ (x, \xi) \in T^*M : x \in \widetilde{V}_j, \,\, |p(x,\xi)
-E| \leq \alpha \}, 
\ee
for some $\alpha>0$.  
Suppose the $\widetilde{V}_j$ satisfy $\dist_g(\widetilde{V}_1, \widetilde{V}_2) =
L,$ and assume 
\begin{eqnarray}
\label{dyn-assumption-10}
\left\{ \begin{array}{l}
\exists C_1,C_2 >0 \text{ such that }
\forall \rho \text{ in a neighbourhood of } V_1, \\
\exp( t H_p)(\rho) \in  V_2 \text{ for } \\
\sqrt{E}(L + C_1) \leq  t
\leq \sqrt{E}(L + C_1 + C_2).
\end{array} \right.
\end{eqnarray}
Suppose $A\in \Psi_h^{0,0}$ is microlocally equal to $1$ in $V_2  $.
If $B \in \Psi_h^{0,0}$ and $\WF (B) \subset V_1$, then there exists a
constant $C>0$ depending only on $C_1, C_2$ such that
\begin{eqnarray*}
\left\| Bu \right\| 
& \leq & C L h^{-1} \|B \|_{\HH \to \HH}  \left\| (P(h)-z)u
\right\| + 2 (E + \alpha)^{3/4} \frac{(C_1+1)}{\sqrt{C_2}} \|B\|_{\HH
  \to \HH} \| Au \|  \\
&& \quad + \O (h) \|\widetilde{B} u\|,
\end{eqnarray*}
where 
\be
\widetilde{B} \equiv 1 \text{ on } \cup_{0 \leq t \leq \sqrt{E}(L + C_1 + C_2)} \exp( t H_p)(
\WF B).
\ee
\end{lemma}

\subsection{The gluing}

Now fix an energy level $z$.  
As described briefly above, we now assume that the function $\Gamma$
can be further decomposed as a sum of pseudodifferential cutoffs:
\[
\Gamma =  \sum_{j = 1}^N \Gamma_j,
\]
where for each $\Gamma_j$, $\Gamma_j = 1$ on a set where $H_p = 0$,
and we have a microlocal black
box estimate of the form
\[
\| \Gamma_j u \| \leq \frac{\alpha_j(h)}{h} \|(P-z) \Gamma_j u \|.
\]
Here, it is necessary that $\alpha_j(h) = \O(h^{-K})$ for the
technique to work (this is the same condition in, for example,
\cite{DaVa-gluing}).  However, in this paper, we have seen that for
the present applications, this 
estimate is true with $K = 1 + \epsilon$.  We will assume this is true
to save a little bit of work later (but we will point out where the
extra step would be needed).  We also require that each
$\Gamma_j \in \s^0$.  Naturally, at least one of the $\Gamma_j$s will
be supported for large frequencies $\xi$ where $P-z$ is elliptic, so
in this case the corresponding $\alpha_j(h)= \O(h)$.

We write for $s<-1/2$ and some positive numbers $c_1, c_2$:
\begin{align*}
\| \rho_{-s} (P-z) u \|^2 & \geq c_1(  \| \rho_{-s}(1-\Gamma) (P-z) u
\|^2 + \| \Gamma (P-z) u \|^2 ) \\
& \geq c_2 (  \| \rho_{-s}(1-\Gamma) (P-z) u
\|^2 + \sum_{j = 1}^N \| \Gamma_j (P-z) u
\|^2 ).
\end{align*}
There is no $\rho_{-s}$ in the terms coming from $\Gamma$ because
$\rho_{-s}$ was assumed to equal 1 on $\supp \Gamma$.  We now write
\begin{align*}
\| \rho_{-s} (1 - \Gamma) (P-z) u \|^2 & = \| \rho_{-s} (P-z) (1 -
\Gamma)  u   + [P,\Gamma] u \|^2 \\
& =  \| \rho_{-s} (P-z) (1 -
\Gamma)  u \|^2 + \|  [P,\Gamma] u \|^2 \\
& \quad + 2 \Re \lll \rho_{-s} (P-z) (1 -
\Gamma)  u   ,  [P,\Gamma]
u  \rrr \\
& \geq \| \rho_{-s} (P-z) (1 -
\Gamma)  u \|^2 + \|  [P,\Gamma] u \|^2 \\
& \quad - 2 \| \rho_{-s} (P-z) (1 -
\Gamma)  u \|  \|  [P,\Gamma]
u  \|.
\end{align*}
Expanding similarly the terms involving the $\Gamma_j$s, we get
\begin{align*}
\| \rho_{-s} (P-z) u \|^2 & \geq \| \rho_{-s} (P-z) (1 -
\Gamma)  u \|^2 + \|  [P,\Gamma] u \|^2 \\
& \quad - 2 \| \rho_{-s} (P-z) (1 -
\Gamma)  u \|  \|  [P,\Gamma]
u  \| \\
& \quad + \sum_{j = 1}^N \big(\| \rho_{-s} (P-z) 
\Gamma_j  u \|^2 + \|  [P,\Gamma_j] u \|^2 \\
& \quad \quad - 2 \| \rho_{-s} (P-z) 
\Gamma_j  u \|  \|  [P,\Gamma_j]
u  \| \big).
\end{align*}
Applying Cauchy's inequality yields for $\eta>0$:
\begin{align*}
& \| \rho_{-s} (P-z) u \|^2 \\
& \quad \geq \| \rho_{-s} (P-z) (1 -
\Gamma)  u \|^2 + \|  [P,\Gamma] u \|^2 \\
& \quad \quad - 2 \left( \eta \| \rho_{-s} (P-z) (1 -
\Gamma)  u \|^2 + 4 \eta^{-1}  \|  [P,\Gamma]
u  \|^2 \right) \\
& \quad \quad + \sum_{j = 1}^N \big( \| \rho_{-s} (P-z) 
\Gamma_j  u \|^2 + \|  [P,\Gamma_j] u \|^2 \\
& \quad \quad \quad - 2 ( \eta \| \rho_{-s} (P-z) 
\Gamma_j  u \|^2 +  4 \eta^{-1} \|  [P,\Gamma_j]
u  \|^2) \big),
\end{align*}
which, by taking $\eta>0$ sufficiently small but fixed yields (for
some positive constant $c_3>0$ and some large constant $C>0$)
\begin{align}
\| \rho_{-s} (P-z) u \|^2 & \geq c_3\| \rho_{-s} (P-z) (1 -
\Gamma)  u \|^2 -C  \|  [P,\Gamma]
u  \|^2 \notag \\
& \quad + \sum_{j = 1}^N \left( c_3\| \rho_{-s} (P-z) 
\Gamma_j  u \|^2 -C \|  [P,\Gamma_j] u \|^2 \right). \label{E:P-lower-cutoff}
\end{align}

We are now in a position to apply Lemma \ref{wf-lemma-2} to each of
the commutator terms.  Let $\Gamma_\star$ be any of the microlocal
cutoffs in \eqref{E:P-lower-cutoff}.  The commutator $[P , \Gamma_\star]$
is of order $h$ and supported in a region where every $H_p$ trajectory
flows out to $\pm \infty$ in space (in at least one direction).  We
apply Lemma \ref{wf-lemma-2} with 
\[
A = \rho_{s}(1 - \Gamma),
\]
the constant $C_2 = M$.  For $\tB$, we choose an appropriate
microlocal 
cutoff $\psi_\star$ for each $j$ (and for $\Gamma$) so that $\tB = \tchi \psi_\star$ is supported
where $H_p \neq 0$ on the flowout of the support of the symbol of
$[P, \Gamma_\star]$ and we get
\[
\| [P , \Gamma_\star] u \| \leq C_M \| (P-z) u \| +
\frac{C_0 h}{\sqrt{M}} \| \rho_{s}(1 - \Gamma) u \| + C h^2 \| \tchi
u \|,
\]
with $C_0>0$ independent of $h$ and $M$ (of course $C_M$ does depend
on $M$ but that is okay for our applications).  Plugging into
\eqref{E:P-lower-cutoff}, we get
\begin{align}
& \| \rho_{-s} (P-z) u \|^2 \notag \\
& \quad \geq c_3\| \rho_{-s} (P-z) (1 -
\Gamma)  u \|^2 \notag \\
& \quad \quad + \sum_{j = 1}^N  c_3\| \rho_{-s} (P-z) 
\Gamma_j  u \|^2 \notag \\
& \quad \quad \quad -C'  ( C_M \| (P-z) u \|^2 +
\frac{C_0 h^2}{{M}} \| \rho_{s}(1 - \Gamma) u \|^2 + C h^4 \sum_{j =
  0}^N \| \tchi \psi_j
u \|^2) . \label{E:P-lower-cutoff-2}
\end{align}
Here the constant $C'>0$ can be quite large as we are summing over all
the commutator terms, but is independent of the parameter $M$.  The sum is from $j = 0$ to $N$, where we
identify $\Gamma_0 = \Gamma$.

For the applications in this paper, the error term $ C h^4 \| \tchi
u \|^2$ is just barely too big, so we apply Lemma \ref{wf-lemma-2} one
more time to this term with the same $A$ and $\tB = \tchi_2$ to get
\[
h^2\| \tchi \psi u \| \leq C_M h\| (P-z) u \| + C_M h^2 \| \rho_{s}
(1-\Gamma) u \| + C h^3 \| \tchi_2 u \|,
\]
where now all the constants may be large.  Note if the $\alpha_j(h) =
\O(h^{-N})$ for a much larger $N$, then we could apply this argument a
finite number of times to further reduce the size of the error.
Plugging into \eqref{E:P-lower-cutoff-2} and absorbing terms which are
small into the larger ones, we get
\begin{align*}
& \| \rho_{-s} (P-z) u \|^2 \notag \\
& \quad \geq c_3\| \rho_{-s} (P-z) (1 -
\Gamma)  u \|^2 \notag \\
& \quad \quad + \sum_{j = 1}^N  c_3\| \rho_{-s} (P-z) 
\Gamma_j  u \|^2 \notag \\
& \quad \quad \quad -C'  ( C_M \| (P-z) u \|^2 +
\frac{C_0 h^2}{{M}} \| \rho_{s}(1 - \Gamma) u \|^2 + C h^6  \| \tchi_2
u \|^2) .
\end{align*}

Finally, we move the negative terms with $(P-z)$ to the left hand side and
apply all the assumed black box microlocal estimates to conclude
\begin{align*}
& \| \rho_{-s} (P-z) u \|^2 \notag \\
& \quad \geq c_3h^2\| \rho_{s} (1 -
\Gamma)  u \|^2 \notag \\
& \quad \quad + \sum_{j = 1}^N  c_3\frac{h^2}{\alpha_j^2(h)} \|  
\Gamma_j  u \|^2 \notag \\
& \quad \quad \quad -C'  ( 
\frac{C_0 h^2}{{M}} \| \rho_{s}(1 - \Gamma) u \|^2 + C h^6  \| \tchi_2
u \|^2) .
\end{align*}
By taking $M>0$ sufficiently large, the term 
\[
-C' \frac{C_0 h^2}{{M}} \| \rho_{s}(1 - \Gamma) u \|^2
\]
can be absorbed into the positive term with the same cutoffs.  After
taking the worst lower bound and summing over our partition of unity, this
gives
\begin{align*}
& \| \rho_{-s} (P-z) u \|^2 \notag \\
& \quad \geq c_4 h^2 \min \left\{ 1, \alpha_1^{-2}(h), \ldots,
  \alpha_N^{-2} (h) \right\} \| \rho_s u \|^2 - C'' h^6  \| \tchi_2
u \|^2.
\end{align*}
We finish by observing that 
\[
h^6 \ll h^2 \min \left\{ 1, \alpha_1^{-2}(h), \ldots,
  \alpha_N^{-2} (h) \right\},
\]
and $\rho_s \equiv 1$ on $\supp \tchi_2$, so this term can be absorbed
as well, leaving us with a final estimate of
\begin{align*}
& \| \rho_{-s} (P-z) u \|^2 \notag \\
& \quad \geq c_4 h^2 \min \left\{ 1, \alpha_1^{-2}(h), \ldots,
  \alpha_N^{-2} (h) \right\} \| \rho_s u \|^2.
\end{align*}

\bibliographystyle{alpha}
\bibliography{inf-deg-bib}

\def\cprime{$'$} \def\cftil#1{\ifmmode\setbox7\hbox{$\accent"5E#1$}\else
  \setbox7\hbox{\accent"5E#1}\penalty 10000\relax\fi\raise 1\ht7
  \hbox{\lower1.15ex\hbox to 1\wd7{\hss\accent"7E\hss}}\penalty 10000
  \hskip-1\wd7\penalty 10000\box7}
\begin{thebibliography}{CdVP94b}

\bibitem[BHW07]{BHW-spread}
Nicolas Burq, Andrew Hassell, and Jared Wunsch.
\newblock Spreading of quasimodes in the {B}unimovich stadium.
\newblock {\em Proc. Amer. Math. Soc.}, 135(4):1029--1037 (electronic), 2007.

\bibitem[BZ04]{BuZw-bb}
Nicolas Burq and Maciej Zworski.
\newblock Geometric control in the presence of a black box.
\newblock {\em J. Amer. Math. Soc.}, 17(2):443--471 (electronic), 2004.

\bibitem[CdVP94a]{CdVP-I}
Y.~Colin~de Verdi{\`e}re and B.~Parisse.
\newblock \'{E}quilibre instable en r\'egime semi-classique.
\newblock In {\em S\'eminaire sur les \'Equations aux D\'eriv\'ees Partielles,
  1993--1994}, pages Exp.\ No. VI, 11. \'Ecole Polytech., Palaiseau, 1994.

\bibitem[CdVP94b]{CdVP-II}
Yves Colin~de Verdi{\`e}re and Bernard Parisse.
\newblock \'{E}quilibre instable en r\'egime semi-classique. {II}. {C}onditions
  de {B}ohr-{S}ommerfeld.
\newblock {\em Ann. Inst. H. Poincar\'e Phys. Th\'eor.}, 61(3):347--367, 1994.

\bibitem[Chr07]{Chr-NC}
Hans Christianson.
\newblock Semiclassical non-concentration near hyperbolic orbits.
\newblock {\em J. Funct. Anal.}, 246(2):145--195, 2007.

\bibitem[Chr08]{Chr-disp-1}
Hans Christianson.
\newblock Dispersive estimates for manifolds with one trapped orbit.
\newblock {\em Comm. Partial Differential Equations}, 33:1147--1174, 2008.

\bibitem[Chr10]{Chr-NC-erratum}
Hans Christianson.
\newblock Corrigendum to ``{S}emiclassical non-concentration near hyperbolic
  orbits'' [{J}. {F}unct. {A}nal. 246 (2) (2007) 145--195].
\newblock {\em J. Funct. Anal.}, 258(3):1060--1065, 2010.

\bibitem[Chr11]{Chr-QMNC}
Hans Christianson.
\newblock Quantum monodromy and nonconcentration near a closed semi-hyperbolic
  orbit.
\newblock {\em Trans. Amer. Math. Soc.}, 363(7):3373--3438, 2011.

\bibitem[CM13]{ChMe-lsm}
Hans Christianson and Jason Metcalfe.
\newblock Sharp local smoothing for manifolds with smooth inflection
  transmission.
\newblock {\em preprint}, 2013.

\bibitem[CW11]{ChWu-lsm}
Hans Christianson and Jared Wunsch.
\newblock Local smoothing for the schr\"odinger equation with a prescribed
  loss.
\newblock {\em Amer. J. Math. to appear}, 2011.

\bibitem[Doi96]{Doi}
Shin-ichi Doi.
\newblock Smoothing effects of {S}chr\"odinger evolution groups on {R}iemannian
  manifolds.
\newblock {\em Duke Math. J.}, 82(3):679--706, 1996.

\bibitem[DV12]{DaVa-gluing}
Kiril Datchev and Andr\'as Vasy.
\newblock Gluing semiclassical resolvent estimates via propagation of
  singularities.
\newblock {\em Int. Math. Res. Not.}, 2012(23):5409--5443, 2012.

\bibitem[H{\"o}r71]{Hor-sing}
Lars H{\"o}rmander.
\newblock On the existence and the regularity of solutions of linear
  pseudo-differential equations.
\newblock {\em Enseignement Math. (2)}, 17:99--163, 1971.

\bibitem[SZ07]{SjZw-frac}
Johannes Sj{\"o}strand and Maciej Zworski.
\newblock Fractal upper bounds on the density of semiclassical resonances.
\newblock {\em Duke Math. J.}, 137(3):381--459, 2007.

\bibitem[Tao06]{Tao-book}
Terence Tao.
\newblock {\em Nonlinear dispersive equations}, volume 106 of {\em CBMS
  Regional Conference Series in Mathematics}.
\newblock Published for the Conference Board of the Mathematical Sciences,
  Washington, DC, 2006.
\newblock Local and global analysis.

\bibitem[Tay81]{Tay-pdo}
Michael~E. Taylor.
\newblock {\em Pseudodifferential operators}, volume~34 of {\em Princeton
  Mathematical Series}.
\newblock Princeton University Press, Princeton, N.J., 1981.

\bibitem[TZ98]{TaZw}
Siu-Hung Tang and Maciej Zworski.
\newblock From quasimodes to resonances.
\newblock {\em Math. Res. Lett.}, 5(3):261--272, 1998.

\bibitem[VZ00]{VaZw}
Andr{\'a}s Vasy and Maciej Zworski.
\newblock Semiclassical estimates in asymptotically {E}uclidean scattering.
\newblock {\em Comm. Math. Phys.}, 212(1):205--217, 2000.

\end{thebibliography}

\end{document}